\documentclass[11pt]{article}
\usepackage[dvipsnames]{xcolor}

\usepackage{amsthm,amsmath,color,natbib}
\usepackage{latexsym, epsfig, amssymb, amsmath, amsthm, graphicx, mathrsfs, booktabs}

\usepackage{arydshln}

\usepackage{setspace}
\usepackage{graphicx}
\newtheorem{definition}{Definition}

\newtheorem{prop}{Proposition}
\newtheorem{lem}{Lemma}
\newtheorem{coro}{Corollary}
\newtheorem{rmk}{Remark}

\newtheorem{assu}{Assumption}

\newtheorem{thm}{Theorem}

\makeatother

\newcommand{\bu}{\mbox{\bf u}}

\newcommand{\bw}{\mbox{\bf w}}
\newcommand{\bx}{\mbox{\bf x}}

\newcommand{\bone}{\mbox{\bf 1}}

\newcommand{\bmu}{\mbox{\boldmath $\mu$}}

\newcommand{\bSig}{\mbox{\boldmath $\Sigma$}}

\newcommand{\var}{\mathrm{var}}

\newcommand{\cov}{\mathrm{cov}}

\newcommand{\tr}{\mathrm{tr}}
\newcommand{\diag}{\mathrm{diag}}

\def\independenT#1#2{\mathrel{\setbox0\hbox{$#1#2$}%
\copy0\kern-\wd0\mkern4mu\box0}}

\newcommand{\p}{{\rm I}\kern-0.18em{\rm P}}
\newcommand{\1}{{\rm 1}\kern-0.24em{\rm I}}
\newcommand{\E}{{\rm I}\kern-0.18em{\rm E}}
\newcommand{\non}{\nonumber \\}
\newcommand{\bbA}{{\bf A}}

\newcommand{\bba}{{\bf a}}
\newcommand{\bbB}{{\bf B}}
\newcommand{\bbC}{{\bf C}}

\newcommand{\bbD}{{\bf D}}

\newcommand{\bbe}{{\bf e}}

\newcommand{\bbF}{{\bf F}}

\newcommand{\bbG}{{\bf G}}

\newcommand{\bbH}{{\bf H}}

\newcommand{\bbw}{{\bf w}}
\newcommand{\bbI}{{\bf I}}
\newcommand{\bbi}{{\bf i}}
\newcommand{\bbj}{{\bf j}}

\newcommand{\bbM}{{\bf M}}

\newcommand{\bbU}{{\bf U}}
\newcommand{\bbu}{{\bf u}}
\newcommand{\bbV}{{\bf V}}
\newcommand{\bbv}{{\bf v}}

\newcommand{\bbW}{{\bf W}}

\newcommand{\bbX}{{\bf X}}

\newcommand{\bbx}{{\bf x}}

\newcommand{\bby}{{\bf y}}

\newcommand{\bbb}{{\bf b}}

\newcommand{\ep}{\ensuremath{\epsilon}}

\usepackage{xcolor}
\usepackage[draft,inline,nomargin,index]{fixme}
\fxsetup{theme=color,mode=multiuser}
\FXRegisterAuthor{yf}{ayf}{\color{magenta} YF}

\usepackage{amsthm,amsmath,amsfonts,amssymb,natbib,mathtools,mathrsfs,algorithm,framed,multirow}
\usepackage[noend]{algpseudocode}
\usepackage{ccaption}
\usepackage{longtable}

\usepackage{enumerate}
\usepackage{verbatim}
\usepackage{subfigure}
\usepackage{bbm}
\usepackage{threeparttable}
\usepackage[colorlinks,linkcolor=blue,citecolor=blue,urlcolor=blue]{hyperref}
\usepackage{xr}

\usepackage{JASA_manu}
\usepackage{authblk}

\makeatother

\begin{document}

%
%

\title{Eigen selection in spectral clustering:  a theory guided practice}

\author[1]{Xiao Han}
\author[2,3]{Xin Tong}
\author[2]{Yingying Fan}
\affil[1]{\footnotesize International Institute of Finance, School of Management, University of Science and Technology of China}
\affil[2]{\footnotesize Department of Data Sciences and Operations, Marshall School of Business, University of Southern California. }
\affil[3]{To whom correspondence should be addressed. xint@marshall.usc.edu}

\date{}

\maketitle

\begin{abstract}
Based on a Gaussian mixture  type model, we derive an eigen selection procedure that improves the usual spectral clustering in high-dimensional settings. Concretely, we derive the asymptotic expansion of the spiked eigenvalues under  eigenvalue multiplicity and eigenvalue ratio concentration results, giving rise to the first theory-backed eigen selection procedure in spectral clustering.  The resulting eigen-selected spectral clustering (ESSC) algorithm enjoys better stability and compares favorably against canonical alternatives. We demonstrate the advantages of ESSC using extensive simulation and multiple real data studies.

{\small \bf KEY WORDS}: clustering, eigen selection, low-rank models, high dimensionality, asymptotic expansions, eigenvectors, eigenvalues.
\end{abstract}

\section{Introduction} \label{Sec1}
Clustering is a widely-used unsupervised learning approach to divide observations into subgroups without the guidance of labels. It is an obvious statistical and machine learning formulation when there are no meaningful labels in the datasets, such as in customer segmentation and criminal cyber-profiling applications. It is also a sensible approach when labels, in theory, do exist, but we have solid reasons to believe that the labels in the datasets are far from accurate. For instance, Medicare-Medicaid fraud detection cannot be formulated as a supervised learning problem, because although the labeled fraudulent transactions are real frauds, people believe that there are a large number of undiscovered frauds in the record.

Over the last sixty years, many clustering approaches have been proposed.
 The most dominant ones include   k-means, hierarchical clustering, spectral clustering, and  various variants  \citep{Hastie.Tibshirani.ea.2009, james2014introduction}. The k-means algorithms  \citep{ bradley1999mathematical, witten2010framework}  adopt a centroid-based clustering approach. Hierarchical clustering algorithms \citep{ ward1963hierarchical} first seek to build a hierarchy of clusters and then make a cut at a hierarchical level.
 Spectral clustering \citep{ ng2002,von2007tutorial} clusters observations using the spectral information of some affinity  matrix derived from the original data for measuring the similarity among observations.

Among the above mentioned main-stream clustering approaches, spectral clustering is particularly well suited for high-dimensional settings, which refers to the situations that the number of features is comparable to or larger than the sample size. High-dimensional settings mainly emerged with modern biotechnologies such as microarray and remain relevant due to the subsequent technological advances such as next-generation sequencing (NGS) technologies. Methodological and theoretical questions in high-dimensional supervised learning (i.e., regression and classification) have been attracting a great deal of attention in the statistics community over the last $20$ years (see the review paper \cite{Zou.2019}  and references within). In contrast, high-dimensional unsupervised problems have had far fewer works so far.  It is a challenging problem mainly because effective dimension reduction is difficult without the assistance of a response variable. Spectral clustering alleviates the problem of curse of dimensionality in high-dimensional clustering by consulting only a few less noisy eigenvectors of an affinity matrix.
For example, suppose that we would like to cluster $n$ observations into $K$ groups, where $K$ is the predetermined cluster number.  Spectral clustering algorithms usually compute the top $K$  eigenvectors of an affinity matrix   and then perform a k-means clustering using just  these $K$ eigenvectors.

The intuition behind  the above spectral clustering method is that under a broad data matrix generative model of low-rank mean matrix plus noise, the data label information is completely captured by the eigenvectors corresponding to top eigenvalues of an affinity matrix based on the low-rank mean matrix. Thus, the  eigenvectors  corresponding to non-spiked eigenvalues can be safely dropped and the purpose of noise reduction is achieved.

In this paper, we formalize the above intuition by  considering  the special case of $K=2$ and Gaussian distributions. Concretely, the data matrix follows the aforementioned structure of low rank (i.e., rank $= 2$) mean matrix plus noise defined as $\bbX=\E\bbX+(\bbX-\E\bbX)$, where $\bbX$ is a $p\times n$ matrix and $n$ is the sample size.   A natural and popular way is to construct the affinity matrix as $\bbX^\top\bbX$ \footnote{A comparison with one alternative affinity matrix construction is given in subsection \ref{sec: comp}}.
We show that the two spiked eigenvectors of $\bbH := (\E\bbX)^\top\E\bbX$, which can be understood as the noiseless version of the affinity matrix, completely capture the label information. We also identify scenarios where exactly one of the two spiked eigenvectors of $\bbH$ is useful for clustering. Here, an eigenvector is useful if its entries take two distinct values, corresponding to the true cluster labels.
Note that the eigenvectors of $\bbH$ are unavailable to us and the spectral clustering is applied to their sample counterparts, that is, the  eigenvectors of the affinity matrix $\bbX^\top\bbX$. These  motivate us to select useful eigenvectors of the affinity matrix in implementing spectral clustering.

Specifically, in this paper, we propose  an innovative \textbf{e}igen \textbf{s}election procedure in the usual \textbf{s}pectral \textbf{c}lustering algorithms and name the resulting algorithm ESSC. Our eigenvector selection step is guided {\color{black}by the theoretical investigation of the top two eigenvectors of $\bbH$.}   We also provide theoretical justification  on our selection criteria.  Our theoretical development does not  require a sparsity assumption on the data generative model, such as those in \cite{cai2013} and \cite{jin2016}. {\color{black}This suggests that our procedure is potentially suitable for a wider range of  applications.}  A by-product of our theoretical development is an asymptotic expansion of the eigenvalues when the population eigenvalues are close to each other (Proposition \ref{t1}). This is a result of stand-alone interest.  We provide extensive simulation studies, and observe that in a vast array of settings, our clustering algorithm ESSC compares favorably in terms of stability and mis-clustering rate against the spectral clustering algorithm without the eigen selection step. These pieces of empirical evidence suggest that ESSC in general, increases the stability of spectral clustering algorithms and achieves competitive clustering results compared with the canonical alternatives. Although our theoretical analysis is conducted under Gaussian distribution assumption, the general idea of eigenvector selection extends to other settings and other high-dimensional clustering problems such as community detection using network data.

We acknowledge that although  the eigen selection idea for spectral clustering is mostly absent in the statistics community, it was practiced in one previous work in the computer science literature. Indeed,
\cite{xiang2008} proposed an EM algorithm to select the eigenvectors of an affinity matrix. But their approach is a heuristic practice and lacks theoretical analysis for the eigenvalues and eigenvectors to support the method.

There is  relatively recent literature on theoretical and methodological developments on high-dimensional clustering. For instance, \cite{ng2002} proposed a symmetric-Laplacian-matrix-based spectral clustering approach and prove the corresponding consistency. \cite{cai2019chime} proposed a clustering procedure based on the EM algorithm for a high-dimensional Gaussian mixture model and proved consistency and minimax optimality for the procedure. \cite{jin2016} proposed a Kolmogorov–Smirnov (KS) score based feature selection approach (IF-PCA) to first reduce the feature dimension before implementing spectral clustering on a centered version of the data. The feature selection idea for clustering was also considered in other works including \cite{yaohall} and \cite{azizyan2013minimax}. None of these aforementioned works select eigenvectors. In this sense, our method and theory complement the existing literature by providing a way to stabilize and improve the performance of existing spectral clustering methods.

The rest of the paper is organized as follows. We introduce the statistical model and key notations in Section \ref{Sec2}. In Section \ref{algo}, we present the main algorithm and detailed rationale that leads to it. Section \ref{Sec2.1} includes the theoretical results. Simulation study and real data analysis are conducted  in Sections \ref{simula} and \ref{sec::real data} respectively, followed by a short discussion.  Technical lemmas, proofs and further discussion are relegated to the {\color{black} Supplementary Material}.

\section{Model setting and notations} \label{Sec2}

In the methodological development and theoretical analysis, we consider the following sampling scheme.  We assume that the data matrix
 $\bbX = (\bx_1, \ldots, \bx_n)$  is generated from

  \begin{equation}\label{0830.1h}
\bx_i=Y_i\bmu_1+(1-Y_i)\bmu_2+\bw_i,\ i=1,\ldots,n\,,
\end{equation}
where $\{\bw_i\}_{i=1}^{n}$ are i.i.d. from $p$-dimensional Gaussian distribution $\mathcal{N}(\bf{0}, \bSig)$, $\bmu_1$, and $\bmu_2$ are two $p$-dimensional non-random vectors, and $Y_1,\ldots, Y_n\in\{0, 1\}$ are deterministic latent class labels.  As such, $Y_i = 1$ means that the $i$th observation $\bx_i$ is from class $1$, and $Y_i = 0$ means that  $\bx_i$ is from class $2$.  The parameters $\bmu_1$, $\bmu_2$ and $\bSig$ are assumed to be unknown. Without loss of generality, we assume that $\bmu_1\neq \bmu_2$ and $\bmu_2\neq 0$.

  The main objective is to recover the latent labels $Y_i$'s from the data matrix $\bbX$.  If $\{Y_i\}_{i=1}^n$ were i.i.d Bernoulli random variables, model \eqref{0830.1h} would be a Gaussian mixture model. Our analysis can extend to that setting but we opt for considering fixed $Y_i$'s to  focus on our attention to the eigen selection principle.

We introduce some notations that will be used throughout the paper. For a matrix $\bbB$, we use $\|\bbB\|$ to denote its spectral norm. For any vector $\bbx$, $\bbx(i)$ represents the $i$-th coordinate of $\bbx$. For any random matrix (or vector) $\bbA$, we use $\E\bbA$ to denote its expectation. We define $c_{11}=\|\bmu_1\|_2^2$, $c_{22}=\|\bmu_2\|_2^2$ and $c_{12}=\bmu_1^{\top}\bmu_2$, where $\|\cdot\|_2$ is the $L_2$ norm of a vector. For any positive sequences $u_n$ and $v_n$, if there exists some positive constant $c$ such that $u_n\ge c v_n$ for all $n\in\mathbb{N}$, then we denote $u_n\gtrsim v_n$.  We denote the $i$-th largest eigenvalue of a square matrix $\bbA$ by $\lambda_i(\bbA)$. Finally, we denote $\sigma_{n}^2=\|\bSig\|^2(n+p)$.

 \section{Algorithm}\label{algo}
In this section, we develop a novel eigen selection procedure that improves the widely used spectral clustering algorithms.
We start our reasoning from the noiseless case.
The entire logic flow of the development process is presented before we introduce the final \textbf{e}igen-\textbf{s}elected \textbf{s}pectral \textbf{c}lustering algorithm (ESSC).

\subsection{Motivation if the signal were known}\label{subsec::preml}

Spectral methods frequently act on the top eigenvectors of the affinity matrix $\bbX^{\top}\bbX$ to recover the underlying latent class labels. As introduced previously, a common practice is to use the top $K=2$ eigenvectors. In this section, we provide some intuition on why the top two eigenvectors contain useful information for clustering.

For notational convenience, denote $\bba_1 = \bby = (Y_1, \ldots, Y_n)^{\top}$  and $\bba_2 = \mathbf{1} - \bby$.      Let $n_1=\|\bba_1\|_2^2$ and $n_2=\|\bba_2\|_2^2$, then $n_1$ and $n_2$ are the numbers of non-zero components of $\bba_1$ and $\bba_2$ respectively, and $n_1+n_2=n$. A noiseless counterpart of $\bbX^{\top}\bbX$ is  $\bbH= (\E\bbX)^{\top} \E \bbX $. By model \eqref{0830.1h},  $\bbH$ can be decomposed by
\begin{equation}\label{0122.1}
 \bbH =\bba_1\bba_1^{\top}c_{11}+\bba_2\bba_2^{\top}c_{22}+\bba_1\bba_2^{\top}c_{12}+\bba_2\bba_1^{\top}c_{12}\ge 0\,.
\end{equation}

Next we discuss the properties of the spectrum of  $\bbH$.
Because
\begin{equation}\label{0416.1}
\text{rank}((\E\bbX)^{\top})\le \text{rank}(\bba_1\bmu_1^{\top})+\text{rank}(\bba_2\bmu_2^{\top})=2\,,
 \end{equation}
 there exist at most two $n$-dimensional orthogonal unit vectors $\bbu_1$ and $\bbu_2$ such that
\begin{equation}\label{0122.1h}
\bbH=d_1^2\bbu_1\bbu_1^{\top}+d_2^2\bbu_2\bbu_2^{\top},\ \text{where} \ \ d_1^2\ge d_2^2\ge 0\,.
\end{equation}
Here, $d_1^2$ and $d_2^2$ are the top two eigenvalues of $\bbH$ and $\bbu_1$ and $\bbu_2$ are the corresponding (population) eigenvectors.  Under our model setting, we have $d_1^2>0$ because otherwise $\bmu_1=\bmu_2 = \bf{0}$, contradicting with the model assumption. For simplicity, in the following, we use $\bbu=(\bbu(1),\ldots,\bbu(n))^{\top}$ to denote either $\bbu_1$ or $\bbu_2$ and $d^2$ to denote its corresponding eigenvalue.
By the definition of eigenvalue,
\begin{equation}\label{0417.1}
\bbH\bbu=d^2\bbu\,.
\end{equation}
Note that $\bbH$ has a block structure by suitable permutation of rows and columns. For example, when $\bba_1=(1,0,1,0)^{\top}$, $\bba_2=(0,1,0,1)^{\top}$, substituting $\bba_1$ and $\bba_2$ into (\ref{0122.1}), we have
$$\bbH=\left(
\begin{array}{cccc}
c_{11} & c_{12} & c_{11} &c_{12}\\
c_{12} & c_{22} & c_{12}&c_{22}\\
c_{11} & c_{12} & c_{11}&c_{12}\\
c_{12} & c_{22} & c_{12}&c_{22}\\
\end{array}
\right)\,.$$
By exchanging the 2nd and 3rd rows and columns of $\bbH$ simultaneously, we can get the following matrix with a clear block structure
 $$\widetilde \bbH=\left(
\begin{array}{cc;{2pt/2pt}cc}
c_{11} & c_{11} & c_{12} &c_{12}\\
c_{11} & c_{11} & c_{12}&c_{12}\\\hdashline[2pt/2pt]
c_{12} & c_{12} & c_{22}&c_{22}\\
c_{12} & c_{12} & c_{22}&c_{22}\\
\end{array}
\right)\,.$$
The eigenvalues of $\bbH$ and $\widetilde \bbH$ are the same and the eigenvectors are the same up to proper permutation of their coordinates.  Inspired by the block structure of $\bbH$ after proper permutation, we can see that (\ref{0122.1}) and (\ref{0417.1}) imply
\begin{equation}\label{0409.1}
c_{11}\sum_{\bba_1(i)=1}\bbu(i)+c_{12}\sum_{\bba_1(i)=0}\bbu(i)=d^2\bbu(j),\ \text{ for }j\text{ such that } \bba_1(j)=1\,,
\end{equation}
\begin{equation}\label{0409.2}
c_{22}\sum_{\bba_1(i)=0}\bbu(i)+c_{12}\sum_{\bba_1(i)=1}\bbu(i)=d^2\bbu(j),\ \text{ for }j\text{ such that } \bba_1(j)=0\,.
\end{equation}
From (\ref{0409.1}) and  (\ref{0409.2}), we conclude that if $d^2>0$, then
\begin{equation}\label{0410.1}
\bba_1(i)=\bba_1(j) \Longrightarrow \bbu(i)=\bbu(j)\,.
\end{equation}
Therefore, the eigenvector $\bbu$ corresponding to a nonzero eigenvalue $d^2>0$ takes at most two distinct values in its components. On the other hand, if $d^2>0$ and $\bbu$  takes two distinct values in its components, then these values have a one-to-one correspondence with the cluster labels. We also notice that when $d^2=0$, $\bbu$ would not be informative for clustering. Given these observations, we introduce the following definition for ease of presentation.

\begin{definition}
	A population eigenvector $\bbu$ is said to have \textit{clustering power} if its corresponding eigenvalue $d^2$ is positive and its coordinates take exactly two distinct values.
\end{definition}

\begin{thm}\label{thm: population-eg}
The top two eigenvalues of $\bbH$ can be expressed as
\begin{equation}\label{0122.7h}
d^2_1=\frac{1}{2}\left(n_1c_{11}+n_2c_{22}+(n_1^2c_{11}^2+n_2^2c_{22}^2+4n_1n_2c_{12}^2-2n_1n_2c_{11}c_{22})^{\frac{1}{2}}\right)\,,
\end{equation}
and
\begin{equation}\label{0122.7ht}
d^2_2=\frac{1}{2}\left(n_1c_{11}+n_2c_{22}-(n_1^2c_{11}^2+n_2^2c_{22}^2+4n_1n_2c_{12}^2-2n_1n_2c_{11}c_{22})^{\frac{1}{2}}\right)\,.
\end{equation}
Moreover, we conclude the following regarding the clustering power of $\bbu_1$ and $\bbu_2$.

\begin{enumerate}[(a)]
\item When $c_{12}^2 = c_{11}c_{22}$, the problem is degenerate with $d_1^2 = n_1c_{11}+n_2c_{22}$ and $d_2^2 = 0$, and only the eigenvector $\bbu_1$ has clustering power.

\item  When $c_{12}^2 \neq c_{11}c_{22}$,  $c_{12}=0$  and $n_1c_{11}=n_2c_{22}$, we face the problem of multiplicity (i.e., $d_1^2 = d_2^2 = n_1c_{11}$) and at least one of $\bbu_1$ and $\bbu_2$ have clustering power.

\item When $c_{12}^2 \neq c_{11}c_{22}$, $c_{12}=0$  and $n_1c_{11}\neq n_2c_{22}$, we have $d_1^2  = \max\{n_1c_{11},n_2c_{22}\}$ and $d_2^2=\min\{n_1c_{11},n_2c_{22}\}>0$, and both $\bbu_1$ and $\bbu_2$ have clustering power.

\item When $c_{12}^2 \neq c_{11}c_{22}$ and $c_{12}\neq 0$, if $n_1c_{11}+n_2c_{12} = n_2c_{22}+n_1c_{12}$, exactly one eigenvector has clustering power, and if $n_1c_{11}+n_2c_{12} \neq  n_2c_{22}+n_1c_{12}$,  both eigenvectors have clustering power.
\end{enumerate}
\end{thm}

Theorem \ref{thm: population-eg} implies that under our model described in equation \eqref{0830.1h}, at least one of $\bbu_1$ and $\bbu_2$ have clustering power.  More importantly, this theorem  indicates that even in the noiseless setting (i.e., when $\bbH$ is known), there are cases in which only one eigenvector has clustering power and that this eigenvector could be either $\bbu_1$ or $\bbu_2$.  This suggests the potential importance of eigenvector selection in spectral clustering and we propose Oracle Procedure \ref{or-pro1} below to select a set $\mathcal U$ of important eigenvectors under the noiseless setting.
\begin{algorithm}[h]
	\caption{[Oracle Procedure 1]}
	\begin{algorithmic}[1]
	\State Set $\mathcal U = \emptyset$.
\State  Check whether $\bbu_1$ has two distinct values in its components. If yes,  add $\bbu_1$ to $\mathcal U$ and go to Step 3; If no, add $\bbu_2$ to $\mathcal U$ and go to Step 5.
	\State  Check whether $d_2^2>0$. If no, go to Step 5; If yes, go to Step 4.
	\State  Check whether $\bbu_2$ has two distinct values in its components. If yes, add $\bbu_2$ to $\mathcal U$ and go to Step 5; if no, go to Step 5.
	\State  Return $\mathcal U$.
	\State Use the eigenvector(s) in $\mathcal U$ for clustering.
\end{algorithmic}\label{or-pro1}
\end{algorithm}

Despite its simple form, Oracle Procedure \ref{or-pro1} is  difficult to implement at the sample level. To elaborate, note that in practice we will have to estimate the eigenvalues and eigenvectors $(d_i^2, \bu_i)$, $i=1,2$. Without loss of generality, assume that $d_1\geq d_2\geq0$.  Note that $d_1$ and $d_2$ are the top two singular values of $\E\bbX$, which can be naturally estimated by the top two singular values of $\bbX$. Further note that $\bbu_1$ and  $\bbu_2$ are the top two right singular vectors of $\E\bbX$, which can be naturally estimated by $\widehat \bbu_1$ and $\widehat \bbu_2$, the top two right singular vectors of $\bbX$. One useful technique in the literature for obtaining these sample estimates is to consider the linearization matrix
$$\mathcal{Z}=\left(
\begin{array}{cc}
0 & \bbX^{\top} \\
\bbX & 0 \\
\end{array}
\right),$$
which is a symmetric random matrix with low-rank mean matrix.   It can be shown that the top two singular values of $\bbX$ are the same as the top two eigenvalues of $\mathcal{Z}$, and the corresponding singular vectors of $\bbX$, after appropriate rescaling, are the subvectors of the top two eigenvectors of $\mathcal{Z}$. See detailed discussions on the relationship in the  subsection \ref{subsec: eigen sel}.

 It has been proved in the literature that for random matrices with expected low rank structure, such as $\mathcal{Z}$,  the estimation accuracy of spiked eigenvectors largely depends on the magnitudes of the corresponding eigenvalues.  Specifically, as shown in \cite{abbe2017entrywise}, the entrywise estimation error for each spiked eigenvector is of order inversely proportional to the magnitude of the corresponding eigenvalue. Thus, dense eigenvector may be estimated very poorly unless the corresponding eigenvalue has a large magnitude, that is, highly spiked.  The results in \cite{abbe2017entrywise} apply to a large Gaussian ensemble matrix with independent entries on and above the diagonal.  Similar conclusions can be found in \cite{FF18} and \cite{bao2018singular} under Wigner or generalized Wigner matrix assumption.

Since spectral clustering is applied to estimated eigenvectors,  the above-mentioned existing results  suggest that in a high-dimensional two-class clustering, one should drop the second eigenvector in spectral clustering if the corresponding eigenvalue is not spiked enough, unless it is absolutely necessary to include it, when, for example, the first spiked population eigenvector has no clustering power.

On the other extreme, if the two spiked eigenvalues are the same, that is, in the case of multiplicity, by part (b) of Theorem \ref{thm: population-eg},  at least one of $\bbu_1$ and $\bbu_2$ has clustering power. We argue that in this situation, at the sample level it is better to use both spiked eigenvectors in clustering for at least two reasons.
First, by Proposition \ref{t1} to be presented in Section \ref{Sec2.1} and the remark after it,  each $d_i$, $i=1, 2$ can only be estimated with accuracy $O_p(1)$. Therefore, detecting the exact multiplicity can be challenging. Second,  the two spiked population eigenvectors are not identifiable. The two spiked sample eigenvectors estimate some rotation of  $(\bbu_1, \bbu_2)$, each with estimation accuracy of order inversely proportional to $d_1$ (or $d_2$) \citep{abbe2017entrywise}.  Thus, even in the worst case where exactly one eigenvector is useful, including both in clustering will not deteriorate the clustering result much because the additional estimation error caused by the useless eigenvector is the same order as caused by the useful eigenvector. In view of the discussions above, we update the oracle procedure as follows. Our implementable algorithm will mimic the  oracle procedure below.

\begin{algorithm}[h]
	\caption{[Oracle Procedure 2]}
	\begin{algorithmic}[1]
	    \State Set $\mathcal U = \emptyset$.
	\State  Check whether $d_1^2/d_2^2<1+c_n$ with $c_n>0$ some threshold depending on $n$ to be specified. If yes, add both $\bbu_1$ and $\bbu_2$ to $\mathcal U$ and go to  Step 4; If no,  go to Step 3.
	\State  Check whether $\bbu_1$ has two distinct values in its components. If yes,  add $\bbu_1$ to $\mathcal U$ and go to Step 4; If no, add $\bbu_2$ to $\mathcal U$ and go to Step 4.
	\State  Return $\mathcal U$.
	\State Use eigenvector(s) in $\mathcal U$ for clustering.
\end{algorithmic}\label{or-pro2}
\end{algorithm}

In step 2 of Oracle Procedure \ref{or-pro2},  a positive sequence $c_n$ is to help check whether $d_1^2$ and $d_2^2$ are close enough. We include a buffer $c_n$ because, in implementation, $d_1$ and $d_2$ are estimated with errors. As discussed above, the rationale behind step 3 is that when the second eigenvalue is much smaller than the first one, and so the  estimated second eigenvector can be too noisy to be included for clustering, we use the estimated second  eigenvector only when the first one is not usable.  Oracle Procedure \ref{or-pro2} prepares us to introduce our final practical selection procedure.

\subsection{Comparison with a centering procedure}\label{sec: comp}
 {\color{black} We digress here to discuss an existing procedure that drops an  eigenvector.  Concretely, a few works, such as IF-PCA,  employ a step to first subtract the mean from the data.  As will be demonstrated next,  this approach reduces the second largest eigenvalue to 0 under our model, and thus always only uses the leading eigenvector for clustering. This can be advantageous under special conditions. However, we will also provide examples where our  approach can be superior. For this reason, we choose not to consider the centering procedure in detail in our paper.
 	
 Let $\bar{\bbx}=\frac{1}{n}\sum_{i=1}^n\bbx_i$ and recall model \eqref{0830.1h}. By subtracting the expectation $\E\bar{\bbx}=\frac{n_1\bmu_1}{n}+\frac{n_2\bmu_2}{n}$, the model becomes
   \begin{equation}\label{0830.1nh}
\bx_i-\E\bar{\bbx}=Y_i\frac{n_2(\bmu_1-\bmu_2)}{n}+(1-Y_i)\frac{n_1(\bmu_2-\bmu_1)}{n}+\bw_i,\ i=1,\ldots,n\,,
\end{equation}
from which we can derive that
$$\text{rank} (\bold{C}) := \text{rank}\left( (\E \bbX-(\E\bar{\bbx})\mathbf{1}_n^\top)(\E\bbX-(\E\bar{\bbx})\mathbf{1}_n^\top)^{\top}\right)=\text{rank}\left(\frac{n_1n_2}{n}(\bmu_1-\bmu_2)(\bmu_1-\bmu_2)^\top\right)=1\,.$$
Hence, the second eigenvalue of $\bold{C}$ is always 0, and the first eigenvalue, denoted by $d_1^2(\bold{C})$, is
\begin{equation}\label{gt1}
  d_1^2(\bold{C})=\frac{n_1n_2\|\bmu_1-\bmu_2\|_2^2}{n}\,.
\end{equation}
}
By comparing \eqref{gt1} with \eqref{0122.7h} and \eqref{0122.7ht}, we see that the effect of the demean step can be complicated.   For one example, if $\bmu_1=-\bmu_2$ and $n_1=n_2$, then $d_1^2=nc_{11}$, $\E\bar\bbx=0$ and $d_1^2(\bold{C})=nc_{11}=d_1^2$. In this case we can see that the demean approach  is desirable. On the other hand, if $c_{12}=0$, then $d_1^2 = \max\{n_1c_{11}, n_2c_{22}\} $ and $d_2^2 = \min\{n_1c_{11}, n_2c_{22}\}$, whereas \eqref{gt1} becomes $ d_1^2(\bold{C})=\frac{ n_1n_2(c_{11}+c_{22})}{n}$, which lies between $d_1^2$ and $d_2^2$. Therefore in this case, the demean approach shrinks the first eigenvalue which reduces the signal strength (cf. the discussion after Oracle Procecdure \ref{or-pro1}).


\subsection{Eigen Selection Algorithm} \label{subsec: eigen sel}
  The two oracle procedures discussed in subsection \ref{subsec::preml} assume the knowledge of $\bbH$.  In practice, we observe $\bbX$ instead of $\bbH$. Next, we will elevate our reasoning on $\bbH$ to that on $\bbX$ and propose an implementable algorithm for eigenvector selection.
 Denote by $\widehat \bbu_1$ and $\widehat \bbu_2$ the eigenvectors of the matrix
  $$\widehat \bbH :=\bbX^\top\bbX\,,$$
  corresponding to the two largest eigenvalues $\widehat t^2_1$ and $\widehat t^2_2$ ($\widehat t_1\ge\widehat t_2\ge 0$) of $\widehat \bbH$, respectively. As discussed after Oracle Procedure \ref{or-pro1}, $\widehat t_1$ and $\widehat t_2$ are the top singular values of $\bbX$, and $d_1$ and $d_2$ are the top singular values of $\E\bbX$. Thus, $\widehat t^2_1$ and $\widehat t^2_2$ estimate $d_1^2$ and $d_2^2$, respectively.
 Further note that $\widehat \bbu_1$ and $\widehat \bbu_2$ are  the top two right singular vectors of $\bbX$,  while $\bbu_1$ and  $\bbu_2$ are the top two right singular vectors of $\E\bbX$.    Under some conditions, when $d_1^2/d_2^2\neq 1$, i.e., no multiplicity, we have $\widehat\bbu_1(i)\approx \bbu_1(i)$ and $\widehat\bbu_2(i)\approx \bbu_2(i)$.   {\color{black}Moreover, when $d_1^2=d_2^2$,} it is only possible for us to show that $(\widehat\bbu_1,\widehat\bbu_2)\approx(\bbu_1,\bbu_2)\widehat\bbU$ (e.g., by Davis-Kahan Theorem), where $\widehat\bbU$ is some $2\times 2$ orthogonal matrix.  Spectral clustering clusters $\bx_i$'s into two groups by dividing the coordinates of $\widehat\bbu_1$ (and$\backslash$or $\widehat\bbu_2$) into two groups via the k-means algorithm. In some scenarios,  $d_2$ is small (compared to $d_1$) and $\widehat \bbu_2$ is significantly disturbed  by the noise matrix $\bbX-\E\bbX$; in these scenarios, $\widehat \bbu_2$  is likely not good enough to distinguish the memberships. Putting these observations together, Oracle Procedure \ref{or-pro2} can be implemented by replacing $(d_i, \bbu_i)$ with the sample version $(\widehat t_i, \widehat \bbu_i)$, $i=1,2$.

As briefly discussed in subsection \ref{subsec::preml}, for easier analysis of the eigenvalues and eigenvectors of $\widehat \bbH = \bbX^\top\bbX$, we consider the linearization matrix $\mathcal {Z}$.  It can be shown that the top two eigenvalues of
$\mathcal{Z}$ are $\widehat t_1$ and $\widehat t_2$.
 Let $\widehat \bbv_1$ and $\widehat\bbv_2$ be the eigenvectors of $\mathcal{Z}$ corresponding to $\widehat t_1$ and $\widehat t_2$ respectively,  and $\widehat\bbv_{-1}$ and $\widehat\bbv_{-2}$ are the eigenvectors of $\mathcal{Z}$ corresponding to $-\widehat t_1$ and $-\widehat t_2$ respectively.  By Lemma \ref{0612-1} in the Supplementary Material, $\pm d_1$ and $\pm d_2$ are the eigenvalues of $\E\mathcal{Z}$, and the vector consisting of the first $n$ entries of the eigenvector of $\E\mathcal{Z}$ corresponding to $d_k$ equals $\frac{\bbu_k}{\sqrt 2}$, $k=1, 2$.
 Moreover, the vector consisting of  the first $n$ entries of the eigenvector of $\mathcal{Z}$ corresponding to $\widehat t_k$ equals $\frac{\widehat \bbu_k}{\sqrt 2}$, $k=1,2$.
 Given these correspondences, we will leverage the two largest eigenvalues of $\mathcal{Z}$ and the corresponding eigenvectors for clustering.

Based on the discussions above, we propose Algorithm \ref{alg1}:  \textbf{E}igen-\textbf{S}elected \textbf{S}pectral \textbf{C}lustering Algorithm (ESSC).
Let $\tau_n$ and $\delta_n$ be two diminishing positive sequences (i.e., $\tau_n+\delta_n=o(1)$) and $\bbu_0$ be an $(n+p)$-dimensional vector in which the first $n$ entries are $1$ and the last $p$ entries are $0$. In numerical implementation, we choose $\tau_n=\log^{-1}(n+p)$ and $\delta_n=\log^{-2}(n+p)$, which are guided by Theorems \ref{t2}--\ref{t3}.  Moreover, let $\mathfrak{f}=n^{-1/2}|\bbu_0^\top\widehat\bbv_1|-2^{-1/2}$.
Note that if all entires of the unit vector $\bbu_1$ are equal, then $|\bbu_0^\top \bbv_1| = |\frac{1}{\sqrt{2}}\bbu_1(1) + \ldots +\frac{1}{\sqrt{2}} \bbu_1(n)| = \left(n/2\right)^{1/2}$, where $\bbv_1$ is the unit eigenvector of $\E\mathcal{Z}$ corresponding to $d_1$.  Hence, checking whether $|\mathfrak{f}|$ is small enough (e.g., $|\mathfrak{f}| < \delta_n$) is a reasonable substitute for checking whether $\bbu_1$ has all equal entries.

 \begin{algorithm}[h]
\caption{[Eigen-Selected Spectral Clustering (ESSC)]}
\begin{algorithmic}[1]
	\State Set $\widehat{\mathcal U} = \emptyset$.
\State Calculate $\widehat t_1$ and $\widehat t_2$ and the corresponding eigenvectors $\widehat \bbv_1$ and $\widehat\bbv_2$ from $\mathcal{Z}$.  Form $\widehat \bbu_1$ and $\widehat \bbu_2$ using the first $n$ entries of $\widehat \bbv_1$ and $\widehat \bbv_2$, respectively.  \label{0505.2}
\State  Check whether ${\widehat t_1}/{\widehat t_2}<1+\tau_n$. If yes, add both $\widehat \bbu_1$ and $\widehat \bbu_2$ to $\widehat{\mathcal U}$ and go to Step 5; if no, go to Step 4.
\State  Check if $|\mathfrak{f}|\ge \delta_n$.  If yes,  add $\widehat\bbu_1$ to $\widehat{ \mathcal U}$ and go to Step 5; if no, add $\widehat \bbu_2$ to $\widehat{\mathcal U}$ and go to Step 5.
\State  Return $\widehat{\mathcal U}$.
\State Apply the $k$-means algorithm to vector(s) in $\widehat{\mathcal U}$ to cluster $n$ instances into two groups.
\end{algorithmic}\label{alg1}
\end{algorithm}

 \section{Theory} \label{Sec2.1}

In this section, we derive a few theoretical results that support the steps 3 and 4 of Algorithm \ref{alg1}.  {\color{black}
 We first prove in Proposition \ref{t1} asymptotic expansions
 for eigenvalues  $\widehat t_1$ and $\widehat t_2$.
 In addition to motivating our handling of multiplicity as discussed in the previous section, these results potentially allow us to design a thresholding procedure on either $\widehat t_1 - \widehat t_2$ or $\widehat t_1 / \widehat t_2$ to detect the multiplicity of eigenvalues.   Indeed,  our proposition fully characterizes the behavior of $\widehat t_1$ and $\widehat t_2$, so that we can derive an expansion for $\widehat t_1 - \widehat t_2$,  but this expansion  depends on the covariance matrix $\bSig$ (see Remark \ref{rem4}), which is not easy  to estimate without the class label information.
 Similarly, an expansion of $\widehat t_1 / \widehat t_2$ would involve $\bSig$.  These concerns motivate us to resort to a less accurate but empirically feasible detection rule for eigenvalue multiplicity. Concretely, we derive concentration results regarding $\widehat t_1 / \widehat t_2$, which do not rely on estimates of $\bSig$ and they give rise to step 3 of Algorithm \ref{alg1}.
} {\color{black}Theorems \ref{t2}--\ref{t3} provide a guarantee for using diminishing positive sequences $\tau_n$ and $\delta_n$  as thresholds for steps 3 and 4 in Algorithm \ref{alg1}.}
We adopt the following assumption in the theory section.

{\color{black}
\begin{assu}\label{a2}
 (i) The eigenvalues of $\bSig$ are bounded away from $0$ and $\infty$. (ii) $n^{1/C}\le p\le n^C$ for some constant $C>0$.
\end{assu}
}

%

Before presenting Proposition \ref{t1}, we will introduce population quantities $t_1$ and $t_2$, which are asymptotically equivalent to population eigenvalues  $d_1$ and $d_2$. We will establish  below that $t_1$ and $t_2$ are indeed the asymptotic means of $\widehat t_1$ and $\widehat t_2$, respectively.  As we work on $\mathcal{Z}$, a linearization of $\widehat \bbH$, we will investigate $\E \mathcal{Z}$ and $\mathcal{Z} - \E \mathcal{Z}$.   Let the eigen decomposition of $\E\mathcal{Z}$ be $$\E\mathcal{Z}=\left[d_1(\bbv_1\bbv_1^\top-\bbv_{-1}\bbv_{-1}^\top)+d_2(\bbv_2\bbv_2^\top-\bbv_{-2}\bbv_{-2}^\top)\right]\,,$$
in which $\bbv_1$ and $\bbv_2$ are the unit eigenvectors  corresponding to $d_1$ and $d_2$, $\bbv_{-1}$ and $\bbv_{-2}$ are the unit eigenvectors corresponding to $-d_1$ and $-d_2$.

Define $\bbV=(\bbv_1,\bbv_2)$, $\bbV_{-}=(\bbv_{-1},\bbv_{-2})$  and $\bbD=\diag(d_1,d_2)$. Then the eigen decomposition of  $\E\mathcal{Z}$ can be written as
\begin{equation}\label{0506.1}
\E\mathcal{Z}=\bbV\bbD\bbV^\top-\bbV_-\bbD\bbV^\top_-\,.
 \end{equation}
Moreover, let
\begin{equation}\label{0506.2}\bbW=\mathcal{Z}-\E\mathcal{Z}=\left(
  \begin{array}{ccc}
  0& (\bbX-\E\bbX)^\top \\
  \bbX-\E\bbX & 0 \\
  \end{array}
\right)\,.
 \end{equation}
For complex  variable $z$, and any matrices (or vectors) $\bbM_1$ and $\bbM_2$ of suitable dimensions, we define the following notations.
 \begin{equation}\label{0214.2}
 \mathcal{R}(\bbM_1,\bbM_2,z)=-\sum_{{l=0,\,l\neq 1}}^L z^{-(l+1)}\bbM_1^\top\E\bbW^l\bbM_2\,,
 \end{equation}

 and
 \begin{equation}\label{eq2}
 f(z)=\left(                 %
 \begin{array}{ccc}   
 f_{11}(z)& f_{12}(z) \\  %
 f_{21}(z) & f_{22}(z) \\  %
 \end{array}
 \right) =\bbI+\bbD\Big(\mathcal{R}(\bbV,\bbV,z)-\mathcal{R}(\bbV,\bbV_{-},z)\big(-\bbD+\mathcal{R}(\bbV_{-},\bbV_{-},z)\big)^{-1}\mathcal{R}(\bbV_{-},\bbV,z)\Big)\,.
 \end{equation}

\begin{lem}\label{0508-1}
Denote by $a_n= d_2-\sigma_n$ and $b_n= d_1+\sigma_n$. 
 Assume that
\begin{equation}\label{0513.2}
  d_1-d_2 = o(\sqrt{d_2}) \ \text{and}\ d_2 \gg \sigma_n^{4/3}\,,
 \end{equation}
 then we have the following conclusions
\begin{itemize}
\item[1.] The equation
\begin{equation}\label{0518.7}
\det(f(z))=0\,,
\end{equation}
in which $f(z)$ is defined in \eqref{eq2},   has at most two solutions in $[a_n,b_n]$. We denote these solutions by $t_1$ and $t_2$ with $t_2\le t_1$.
\item[2.]
\begin{equation}\label{0518.8}
t_k-d_k=O\left(\frac{\sigma_n^2}{d_2}\right), \ k=1,2\,.
\end{equation}
\end{itemize}

\end{lem}

Equation \eqref{0513.2} is a signal strength assumption requiring that the top two eigenvalues should be spiked enough, and that the second eigenvalue cannot be much smaller than the top eigenvalue.  In fact, \eqref{0513.2} implies that $d_1/d_2 \rightarrow 1$, that is, close to multiplicity. Under such conditions, Lemma \ref{0508-1} guarantees the existence of $t_1$ and $t_2$. Moreover, this lemma provides a guarantee that $\frac{t_1}{d_1}$ and $\frac{t_2}{d_2}$ are asymptotically close to  $1$.
The following proposition is established by carefully analyzing the behavior of $\widehat t_k$ around $t_k$, $k=1,2$.
\begin{prop}\label{t1}
 Under  Assumption \ref{a2} and \eqref{0513.2},  we have

	\begin{equation}\label{th1}
	\widehat t_1-t_1=\frac{1}{2}\left[-g_{11}(t_1)-g_{22}(t_1)+\left\{\left(g_{11}(t_1)+g_{22}(t_1)\right)^2-4\left(g_{11}(t_1)g_{22}(t_1)-g^2_{12}(t_1)\right)\right\}^{\frac{1}{2}}\right]+o_{p}(1)\,,
	\end{equation}
	\begin{equation}\label{th2}
	\widehat t_2-t_2=\frac{1}{2}\left[-g_{11}(t_2)-g_{22}(t_2)-\left\{\left(g_{11}(t_2)+g_{22}(t_2)\right)^2-4\left(g_{11}(t_2)g_{22}(t_2)-g^2_{12}(t_2)\right)\right\}^{\frac{1}{2}}\right]+o_{p}(1)\,,
	\end{equation}
	where $g_{11}, g_{12}, g_{21}$ and $g_{22}$ are defined in
	\begin{equation}\label{0227.1}
	g(z)=\left(
	\begin{array}{ccc}
	g_{11}(z)& g_{12}(z) \\
	g_{21}(z) & g_{22}(z) \\
	\end{array}
	\right)=z^2\bbD^{-1}f(z)-\bbV^\top\bbW\bbV\,.
	\end{equation}
	\noindent For $\widehat t_2$, we also have an alternative expression
	\begin{equation}\label{eq:t2hat_t1}
	\widehat t_2-t_1=\frac{1}{2}\left[-g_{11}(t_1)-g_{22}(t_1)-\left\{\left(g_{11}(t_1)+g_{22}(t_1)\right)^2-4\left(g_{11}(t_1)g_{22}(t_1)-g^2_{12}(t_1)\right)\right\}^{\frac{1}{2}}\right]+o_{p}(1)\,.
	\end{equation}
\end{prop}

Proposition \ref{t1} provides asymptotic expansions of $\widehat t_k$ around $t_k$ ($k = 1, 2$) that are not achievable by routine application of the Weyl's inequality.  Indeed,  Proposition \ref{t1} implies that the fluctuations of $\widehat t_k$ around $t_k$ is $O_p(1)$ (c.f., Lemma \ref{0625-1} in the Supplementary Material), while the Weyl's inequality gives  $|\widehat t_k-d_k|\le \|\bbW\|$, which, combined with  Lemma \ref{bod} in the Supplementary Material, implies that the fluctuation of $\widehat t_1-\widehat t_2$ around $d_1 - d_2$ is $O_p(\sigma_n)$. On the other hand, Proposition \ref{t1} also suggests that designing a statistical procedure by thresholding $\widehat t_1 - \widehat t_2$ would be a difficult task, as argued in detail in Remark \ref{rem4}.

\begin{rmk}\label{rem4}
	
	Equations \eqref{th1} and \eqref{eq:t2hat_t1} imply that
	\begin{equation}\label{0826.1h}
	\widehat t_1-\widehat t_2=\left\{\left(g_{11}(t_1)+g_{22}(t_1)\right)^2-4\left(g_{11}(t_1)g_{22}(t_1)-g^2_{12}(t_1)\right)\right\}^{\frac{1}{2}}+o_{p}(1)\,.
	\end{equation}
	To bound the main term in \eqref{0826.1h}, we calculate the variance and covariance of $\bbv_i\bbW\bbv_j$, $1\le i,j\le 2 $,  as follows.
	\begin{equation}\label{th3}
	\var(\bbv_i^\top\bbW\bbv_i)=4\bbw_i^\top\bSig\bbw_i, \ i=1,2\,,
	\end{equation}
	$$\var(\bbv_1^\top\bbW\bbv_2)=\bbw_1^\top\bSig\bbw_1+\bbw_2^\top\bSig\bbw_2, \ i=1,2\,,$$
	$$\cov(\bbv_i^\top\bbW\bbv_i,\bbv_1^\top\bbW\bbv_2)=2\bbw_1^\top\bSig\bbw_2, \ i=1,2\,,$$
	where $\bbw_i$ is the last $p$ entries of $\bbv_i$.
	Also note that
	\begin{equation}\label{0826.2h}
	\E\bbW^2=\diag(n\bSig,\tr\bSig)\,.
	\end{equation}
	Hence,
	$\bbv_1^\top\E\bbW^2\bbv_1-\bbv_2^\top\E\bbW^2\bbv_2=n(\bbw_1\bSig\bbw_1-\bbw_2\bSig\bbw_2) \text{ and } \bbv_1^\top\E\bbW^2\bbv_2=n\bbw_1\bSig\bbw_2.$
	By Lemma \ref{mos} in the Supplementary Material and \eqref{0214.2}, we have
	\begin{equation}\label{th4}
	\bbv_1^\top\E\bbW^2\bbv_1=\frac{1}{2}(n\bbw_1^\top\bSig\bbw_1+tr\bSig)\sim \sigma_n^2\,.
	\end{equation}
	By \eqref{0826.2h} and Assumption \ref{a2} on $\bSig$, for $\bbM_1$ and  $\bbM_2$ with finite columns and spectral norms, we have
	\begin{equation}\label{0826.3h}
	\|\mathcal{R}(\bbM_1,\bbM_2,t_1)+\sum_{{l=0,\,l\neq 1}}^2 t_1^{-(l+1)}\bbM_1^\top\E\bbW^l\bbM_2\|=O\left(\frac{\sigma_n^3}{t_1^2}\right)\,.
	\end{equation}
Then \eqref{th4},  \eqref{0826.3h}, Assumption \ref{a2} and the definition of $g(z)$ together imply that
\begin{equation}\label{0826.5h}
	\left|g_{ij}(t_1)-\frac{t_1^2}{d_i}-\frac{t_1^2}{d_i}\bbv_i^T\bbW\bbv_j+t_1+\frac{\bbv_i^T\E\bbW^2\bbv_j}{d_i}\right|=O\left(\frac{\sigma_3}{t_1^2}\right)\ll \frac{\bbv_1^T\E\bbW^2\bbv_1}{t_1}\,.
	\end{equation}
By Lemma \ref{0508-1} we have $t_1=d_1+O(\frac{\sigma_n^2}{d_2})$, \eqref{0826.5h} suggests that we have with probability tending to 1,
\begin{align}\label{c3}
	&\left\{\left(g_{11}(t_1)+g_{22}(t_1)\right)^2-4\left(g_{11}(t_1)g_{22}(t_1)-g^2_{12}(t_1)\right)\right\}^{\frac{1}{2}}\\
	&\le \left\{\left(\frac{t_1^2(d_1-d_2)}{d_1d_2}+\frac{\bbv_1^\top\E\bbW^2\bbv_1-\bbv_2^\top\E\bbW^2\bbv_2}{t_1}+\bbv_1^\top\bbW\bbv_1-\bbv_2^\top\bbW\bbv_2\right)^2+4\left(\frac{\bbv_1^\top\E\bbW^2\bbv_2}{t_1}+\bbv_1^\top\bbW\bbv_2\right)^2\right\}^{\frac{1}{2}}\\
	&+\ep \frac{\bbv_1^\top\E\bbW^2\bbv_1}{t_1}\,,\nonumber
\end{align}
for any positive constant $\ep$. Through \eqref{th3} and \eqref{th4}, we see that on both sides of \eqref{c3}, the information of $\bSig$ plays an important role. Therefore, a good thresholding procedure on $\widehat t_1-\widehat t_2$ would involve an accurate estimate of $\bSig$, which is difficult to obtain in the absence of label information.
\end{rmk}

Similar to the asymptotic expansion for $\widehat t_1 - \widehat t_2$, an asymptotic expansion for $\widehat t_1 / \widehat t_2$ would also involve the covariance matrix $\bSig$.  Nevertheless, the latter has better concentration property compared to the former, which  motivates us to consider a non-random thresholding rule on $\widehat t_1 / \widehat t_2$. The concentration property of $\widehat t_1 / \widehat t_2$ under different population scenarios is summarized  in  Theorem \ref{t2} and the first part of Theorem \ref{t3}, respectively, with the former corresponding to the case close to multiplicity and the latter corresponding to the case away from multiplicity.  Moreover, the second part of Theorem \ref{t3} validates the step 4 of ESSC. We would like to emphasize that Theorem \ref{t3} does not require $d_2$ to be spiked and thus can be applied even when $d_2=0$.
\begin{thm}\label{t2}
In addition to Assumption \ref{a2}, further assume that $d_1\gg \sigma_n$, $d_1/d_2\le 1+n^{-c}$ for all $n\ge n_0$,   where $c$ and $n_0$ are positive constants, then there exists a positive constant $C$ such that as $n\rightarrow \infty$,
\begin{equation}
\p\left(\frac{\widehat t_1}{\widehat t_2}\ge 1+C\left(\frac{1}{n^c}+\frac{\sigma_n}{d_1}\right)\right)\rightarrow 0\,.
\end{equation}

\end{thm}

\begin{thm}\label{t3}
Let $\bbu_0$ be an $n+p$ vector in which the first $n$ entries are $1$'s and the last $p$ entries are $0$'s. In addition to Assumption \ref{a2}, further assume that $d_1\gg \sigma_n$ and $d_1/d_2\ge 1+c$  for some positive constant $c$.  Then for any positive constant $D$, we have
\begin{equation}\label{0605.1}
\p\left(\frac{\widehat t_1}{\widehat t_2}\ge 1+\frac{c}{2}\right)\ge 1-n^{-D}\,,
\end{equation}
for all $n\ge n_0$, where $n_0$ is some constant that only depends on the  constant $D$. Moreover, if the first $n$ entries of $\bbv_1$ are equal, we have for all $n \geq n_0$,
\begin{equation}\label{0605.2}
\p\left(\left|\left(\frac{1}{n}\right)^{\frac{1}{2}}|\bbu_0^\top\widehat\bbv_1|-\left(\frac{1}{2}\right)^{\frac{1}{2}}\right|\le \sqrt{\frac{2\sigma_n}{d_1}}\right)\ge 1-n^{-D}\,.
\end{equation}

\end{thm}
{\color{black} We note that Theorems \ref{t2} and \ref{t3} require $d_1 \gg \sigma_n$, which is weaker than the condition for $d_1$ in  Proposition \ref{t1}. } {\color{black}  By Theorems \ref{t2} and \ref{t3}, we can choose $\tau_n$ and $\delta_n$ for Algorithm \ref{alg1} such that $C(n^{-c}+ \sigma_n/d_1) \le \tau_n\le c/2$ and $\delta_n\ge \sqrt{ 2\sigma_n/d_1}$ . In our simulation, we let $\tau_n=\log^{-1}(n+p)$ and $\delta_n=\log^{-2}(n+p)$. These choices were reasonable   when  $\log^{-4}(n+p)\ge 2\sigma_n/d_1$ for sufficiently large $n$ and $p$.}


We next discuss that when $p \sim n$, the results in Theorems \ref{t2}--\ref{t3} apply as long as clustering is possible. Concretely, note that both these theorems require that $d_1$, a measure of the difficulty in clustering, to satisfy $d_1\gg \sigma_n$, which reduces to $d_1\gg \sqrt{n}$ when $p\sim n$. In the Supplementary Material,  we establish the clustering lower bound  in Theorem \ref{lowerbound} by showing that  if $d_1\ll \sqrt{n}$, then clustering is impossible regardless of what method to use; see the Supplementary Material for specific assumptions. We further prove in Theorem \ref{exactrec} and Corollary \ref{cor12} in the Supplementary Material that when $d_1\ge \sqrt{2(1+\ep_0)n\log n}$ for any positive constant $\ep_0$, a simple clustering method based on the signs of selected eigenvector can perfectly recover the class labels with probability tending to 1 (i.e., exact recovery). Our exact recovery result is similar to Theorem 3.1 of \cite{abbe2017entrywise}, who studied symmetric random matrices with independent entries on and above diagonals and low expected rank. Moreover, in related papers working on different models such as $\mathbb{Z}_2$-synchronization \citep{bandeira2017tightness} and stochastic block model \citep{abbe2017entrywise}, it is shown that when $d_1(\bbA)$ is at least of order  $\sqrt{n\log  n}$, there exists an exact recovery approach to identify the memberships, where $\bbA$ is the data matrix in the respective context.

%

\section{Simulation Studies}\label{simula}

 In this section, we compare our newly proposed eigen-selected spectral clustering (ESSC) with k-means, Spectral Clustering, CHIME, IF-PCA and the oracle classifier (a.k.a, Bayes classifier). Recall that the oracle classifier to distinguish $\bbx|(Y=1)\sim  N(\bmu_1, \bSig)$ from {\color{black}$\bbx|(Y=0)\sim  N(\bmu_2, \bSig)$} is
 \begin{equation}\label{0424.1}g(\bbx)=\begin{cases}
1, & \text{ if } (\bbx-\frac{\bmu_1+\bmu_2}{2})^\top\bSig^{-1}(\bmu_1-\bmu_2)\ge\log(\frac{\pi}{1-\pi})\,,\cr 0, & \text{ if } (\bbx-\frac{\bmu_1+\bmu_2}{2})^\top\bSig^{-1}(\bmu_1-\bmu_2)<\log(\frac{\pi}{1-\pi})\,,\end{cases}
\end{equation}
 where $\pi = \p(Y=1)$. We generate $n$ i.i.d. copies of $\bbx\sim\pi N(\bmu_1,\bSig)+(1-\pi)N(\bmu_2,\bSig)$ with
 $\pi = 0.5$. We have also experimented with $\pi=0.4$ and the results are very similar so omitted. Throughout this section, we set $\bmu_1=r(\bmu_{11}^{\top},\bmu_{12}^{\top})^{\top}$, where $\bmu_{11}$ is an $l$-dimensional vector in which all entries are $1$, $\bmu_{12}$ is a $(p-l)$-dimensional vector in which all entries are $0$, and $r$ is a scaling parameter. Our simulation is based on the following five models.

\begin{itemize}
\item Model 1: $\bmu_2=\mathbf{0}$,  $n=200$, $p\in\{100,200,400,600,800,1000,1200\}$, $l=15$ and $r=2$. The covariance matrix $\bSig=(\sigma_{ij})$ is symmetric with $\bSig_{ij}=0.8^{|i-j|}$.
\item Model 2: $\bmu_2=r(\bmu_{12}^{\top},\bmu_{11}^{\top})^{\top}$, $n=100$,   $p\in\{100,200,400,600,800,1000,1200\}$, $l=12$ and $r=2$. The covariance matrix $\bSig=r^2 \bbI$.
\item Model 3:  $\bmu_2=\bmu_1 / 2$, $n=200$,   $p\in\{100,200,400,600,800,1000,1200\}$,  $l=60$ and $r=1$. The covariance matrix $\bSig=\bbI$.
\item Model 4: the same as Model 3 except for $p\in\{30,50,100,200,400,600,800\}$ and $l=30$.

\item Model 5: $\bmu_2=1/r(\bmu_{21}^{\top},\bmu_{22}^{\top})^{\top}$, where $\bmu_{21}$ is an $(l/2)$-dimensional vector in which all entries are $1$, $\bmu_{22}$ is a $(p-l/2)$-dimensional vector in which all entries are $0$, $l=20, \ p=400$,   $n\in\{200,400,600,800,1000\}$ and $r=1$. The covariance matrix $\bSig=r^2 \bbI$.
\end{itemize}

In Model 1, the covariance matrix $\bSig$ has non-zero off-diagonal entries. In Models 2--4, each non-zero entry of $\bmu_1$ and $\bmu_2$ with magnitude not bigger than $r$ is covered by Gaussian noise with variance $r^2$.
In Models 3--4, $\bmu_1$ is parallel to $\bmu_2$. {\color{black} With Model 5, we investigate how the trend of the misclustering rate changes with $n$.}

For CHIME , we use the Matlab codes uploaded to \verb+Github+ by the authors of \cite{cai2013}. Since CHIME involves an EM algorithm, the initial value is very important. We use the default initial values provided in the Matlab codes. We also need to provide the other initial values of $\bmu_1$, $\bmu_2$, $\beta_0=\bSig^{-1}(\bmu_1-\bmu_2)$ and $\pi$ denoted by $\widehat\bmu_1$, $\widehat\bmu_2$, $\widehat\beta_0$ and $\widehat\pi$ respectively. Specifically, we set $\widehat\bmu_1=\frac{\sum_{1\le i\le n,Y_i=1}\bbx_i}{n_1}$ and {\color{black}$\widehat\bmu_2=\frac{\sum_{1\le i\le n,Y_i=0}\bbx_i}{n_2}$}, $\widehat\beta_0=\bSig^{-1}(\widehat\bmu_1-\widehat\bmu_2)$  and $\widehat\pi=0.4$.  For Spectral Clustering,  there are a lot of variants. In the simulation part,  we  follow \cite{ng2002} with the common non-linear kernel $k(\bbx,\bby)=\exp\{-\frac{\|\bbx-\bby\|_2^2}{2p}\}$ to construct an affinity matrix.  For IF-PCA in \cite{jin2016}, we directly apply the Matlab code provided by the authors without modification.

We repeat 100 times for each model setting and calculate the average misclustering rate and the corresponding standard error in  Tables \ref{tab1.5}-\ref{tab5.5}.

\begin{table}[htbp]
	\centering
	\caption{The misclustering rate of several approaches for Model 1 with $\pi=0.5$}
	\begin{tabular}{c|cccccc}
		\toprule
    p  & \text{ESSC}      & \text{k-means} &\text{Spectral Clustering}      & \text{CHIME} & \text{IF-PCA} &\text{Oracle}  \\
     	\midrule
100     &  .067(.0017) & .069(.0018)&.071(.0017)&.036(.0045)&.14(.0112) &.002(.0009)  \\
200     &  .072(.0017)  & .074(.0019)&.076(.0019)&.071(.0097)&.15(.0131)  &.002(.001)   \\
400     &  .073(.0021) &.079(.0022)  &.081(.0021)&.088(.0125)&.191(.0137)  &.002(.0009) \\
600     &    .078(.002)  & .088(.0022)&.091(.0022)&.067(.0105)&.21(.0146) &.002(.001)   \\
800     &  .078(.0018) &.1(.0055) &.099(.0023)&.036(.0047)&.258(.0157)  &.002(.001) \\
1000     &  .084(.002) &.117(.0063) &.108(.0026)&.024(.0046)&.257(.0149)  &.002(.0009) \\
1200     &  .087(.0022) &.12(.0053) &.117(.003)&.021(.005)&.266(.0147)  &.002(.0009) \\
		\bottomrule
	\end{tabular}\label{tab1.5}
	\label{tab:addlabel}%
\end{table}%

\begin{table}[htbp]
	\centering
	\caption{The misclustering rate of several approaches for Model 2 with $\pi=0.5$}
	\begin{tabular}{c|cccccc}
		\toprule
    p  & \text{ESSC}      & \text{k-means} &\text{Spectral Clustering}      & \text{CHIME} & \text{IF-PCA} &\text{Oracle}   \\
     	\midrule
100     & .012(.0011) & .011(.001)&.083(.013)&.004(.0006)&.224(.0139) & .008(.0008) \\
200      & .023(.0016)  & .024(.004)&.169(.015)&.002(.0004)& .269(.0139)& .007(.0008)   \\
400      & .042(.0029) &.04(.0049)  &.298(.013)&0(0)& .335(.0124) & .009(.0009) \\
600      & .068(.0034) &.089(.0103)  &.352(.0096)&0(0)& .373(.0107) & .007(.0007) \\
800      & .086(.0037)  &.122(.0121) &.386(.0073)&0(0)& .401(.0088) & .006(.0007)  \\
1000    & .117(.0057)  &.211(.0145)  &.386(.0078)&0(0)&.423(.0076) & .008(.001)  \\
1200      & .16(.0084)  &.238(.0142) &.398(.0069)&0(0)& .407 (.0071)& .006(.0009)  \\
		\bottomrule
	\end{tabular}\label{tab2.5}
	\label{tab:addlabel}%
\end{table}%

\begin{table}[htbp]
	\centering
	\caption{The misclustering rate of several approaches for Model 3  with $\pi=0.5$}
	\begin{tabular}{c|cccccc}
		\toprule
	p  & \text{ESSC}      & \text{k-means} &\text{Spectral Clustering}      & \text{CHIME} & \text{IF-PCA} &\text{Oracle}   \\
     	\midrule
100       & .028(.0012) & .037(.0014)&.038(.0014) & .093(.0121)& .203(.0096)& .028(.0012) \\
200       & .028(.0011)  & .047(.0014) & .049(.0013)& .438(.0117)&.285(.0117) & .026(.0012)   \\
400     & .027(.001) &.085(.0075)    &.073(.0023) & .446(.0106)&.366(.0107) & .026(.001) \\
600     & .032(.0014) &.137(.011)    &.1(.0023) & .468(.0049)& .393(.0088)& .025(.0012) \\
800      & .033(.0013)  &.193(.011)   &.134(.0034) &.442(.0109) &.41(.008) & .029(.0012)  \\
1000     & .033(.0015)  &.269(.0127)   & .161(.004)&.457(.0082) &.424(.0066) & .026(.0012)  \\
1200    & .037(.0013)  &.322(.0114)   &.196(.0059) & .365(.0118)& .425(.0071) & .026(.0011)  \\
		\bottomrule
	\end{tabular}\label{tab3.5}
	\label{tab:addlabel}%
\end{table}%
In general, ESSC deteriorates much slower than k-means as $p$ increases and is more stable than k-means. {\color{black}Tables \ref{tab1.5}--\ref{tab2.5} indicate that k-means is comparable to ESSC when $p$ is small, while ESSC works better than k-means  when $p$ is large.
   For Model 3 in Table \ref{tab3.5}, ESSC outperforms k-means.}  Since the number of non-zero coordinates of $\bmu_1$ and $\bmu_2$ in Model 4 is much fewer than that in Model 3,  the signal strength of the means in Model 4 is not strong enough to have large spiked singular values. As such, the performance of ESSC in Table \ref{tab4.5} is worse than that of k-means when $p$ is smaller (e.g., less than $200$). {\color{black} However, since the misclustering rate of ESSC increases slowly as $p$ increases,} when $p$ passes $200$, ESSC competes favorably against k-means. Comparing to Spectral Clustering, ESSC excels in all models {\color{black} for almost all $p$ and $n$}.  Tables \ref{tab1.5}--\ref{tab2.5} indicate that CHIME outperforms the other approaches for Models 1--2. While for Models 3--4, the performance of CHIME is worse than the others. We conjecture that such a phenomenon happens because the differences of $\bmu_1$ and $\bmu_2$ are small and $\bmu_1 - \bmu_2$ has more non-zero coordinates than that in Model 2, which does not cater the sparse assumptions in CHIME very well. Table \ref{tab5.5} for Model 5 indicates how the misclustering rates change as $n$ increases. When $n$ is small,  We also observe that ESSC performs better than other methods.


\begin{table}[htbp]
	\centering
	\caption{The misclustering rate of several approaches for Model 4  with $\pi=0.5$}
	\begin{tabular}{c|cccccc}
		\toprule
	p  & \text{ESSC}      & \text{k-means} &\text{Spectral Clustering}      & \text{CHIME} & \text{IF-PCA} &\text{Oracle}   \\
     	\midrule
30        & .19(.003) & .105(.0023)&.103(.002) &.47(.0024)&.235(.0055)& .087(.0021) \\
   50    & .2(.0033)  & .112(.003) &.111(.0026) &.472(.0021)&.301(.0083)&.088(.0019)   \\
 100      & .21(.003) &.145(.0059)    &.133(.0029)&.474(.002)&.341(.009)& .084(.0018) \\
 200      & .21(.0028) &.24(.0107)    &.182(.0048)&.474(.0022)&.419(.0065)& .086(.0018) \\
 400    & .23(.0031)  &.372(.008)   &.279(.0079)&.471(.0019)&.448(.0041)& .086(.0019)  \\
 600     & .241(.0034)  &.41(.006)   &.348(.0075) &.47(.0023)&.452(.004)&.086(.002)  \\
800      & .255(.0034)  &.419(.0059)   &.349(.0071)&.473(.0021)&.46(.0026)& .088(.002)  \\
		\bottomrule
	\end{tabular}\label{tab4.5}
	\label{tab:addlabel}%
\end{table}%

\begin{table}[htbp]
	\centering
	\caption{The misclustering rate of several approaches for Model 5  with $\pi=0.5$}
	\begin{tabular}{c|cccccc}
		\toprule
	n  & \text{ESSC}      & \text{k-means} &\text{Spectral Clustering}      & \text{CHIME} & \text{IF-PCA} &\text{Oracle}   \\
     	\midrule
 200      & .04(.0015) &.073(.0058)    &.347(.0096)&.079(.0007)&.384(.0108)& .014(.0009) \\
 400    & .033(.0009)  &.042(.0012)   &.191(.0137)&.016(.0006)&.305(.0133)& .015(.0006)  \\
 600     & .03(.0007)  &.036(.0008)   &.062(.0067) &.022(.0007)&.288(.0139)&.013(.0004)  \\
800      & .029(.0007)  &.032(.0007)   &.037(.0021)&.029(.0006)&.291(.0147)& .013(.0004)  \\
1000      & .029(.0005)  &.031(.0005)   &.033(.0008)&.034(.0006)&.28(.0154)& .014(.0004)  \\
		\bottomrule
	\end{tabular}\label{tab5.5}
	\label{tab:addlabel}%
\end{table}%

%
%
%
%
%
%
%
%
%
%
%
%
%
%
%

\section{Real data analysis}\label{sec::real data}
In this section, we run several real data sets in finance and biomedical diagnosis to compare the newly proposed ESSC with the other clustering approaches.
\subsubsection{Financial data}
{\color{black}We  consider a credit card dataset in \cite{dataset}. This dataset contains transactions made by credit cards in September $2013$ by European cardholders.
Each instance in the data contains $30$ features and the data has labeled $492$ frauds out of $284,807$ transactions. Among these features, $28$ are engineered features obtained from some original features (which are not revealed for privacy concerns), while the other two features are `Time' and `Amount'. We only use the $28$ engineered features to do clustering. Clearly, the data set is highly imbalanced: the fraud transactions account for $0.172\%$ of all transactions. We choose the first $50$ fraud transactions and  the first $5r$ normal transactions, where $r\in \{10,11,\ldots,50\}$. Note that for $r=10$, the fraud and normal groups are balanced in size, and for $r = 50$, normal transactions are 5 times as many as the fraud ones.  On these data sets, we compare ESSC with IF-PCA and two other spectral clustering methods. The first spectral method (SC1) directly applies k-means to the first $n$ rows of $(\widehat \bbv_1,\widehat \bbv_2)$ and the second method (SC2) is the one that uses a non-linear kernel as described in the simulation section. We do not report the performance of CHIME in real data analysis, as initializations on parameters such as $\Sigma$ are not communicated in the original paper and unlike simulation, there is no obvious initialization choice for real data studies. Figure \ref{f6} demonstrates that ESSC is the preferred approach for all $r$'s (i.e., imbalanced ratios), demonstrating the efficiency and stability of ESSC on this financial data set.}
\begin{figure}\label{f6}
\centering
\textbf{Misclusering rate of Credit card data}
		\includegraphics[scale=0.70]{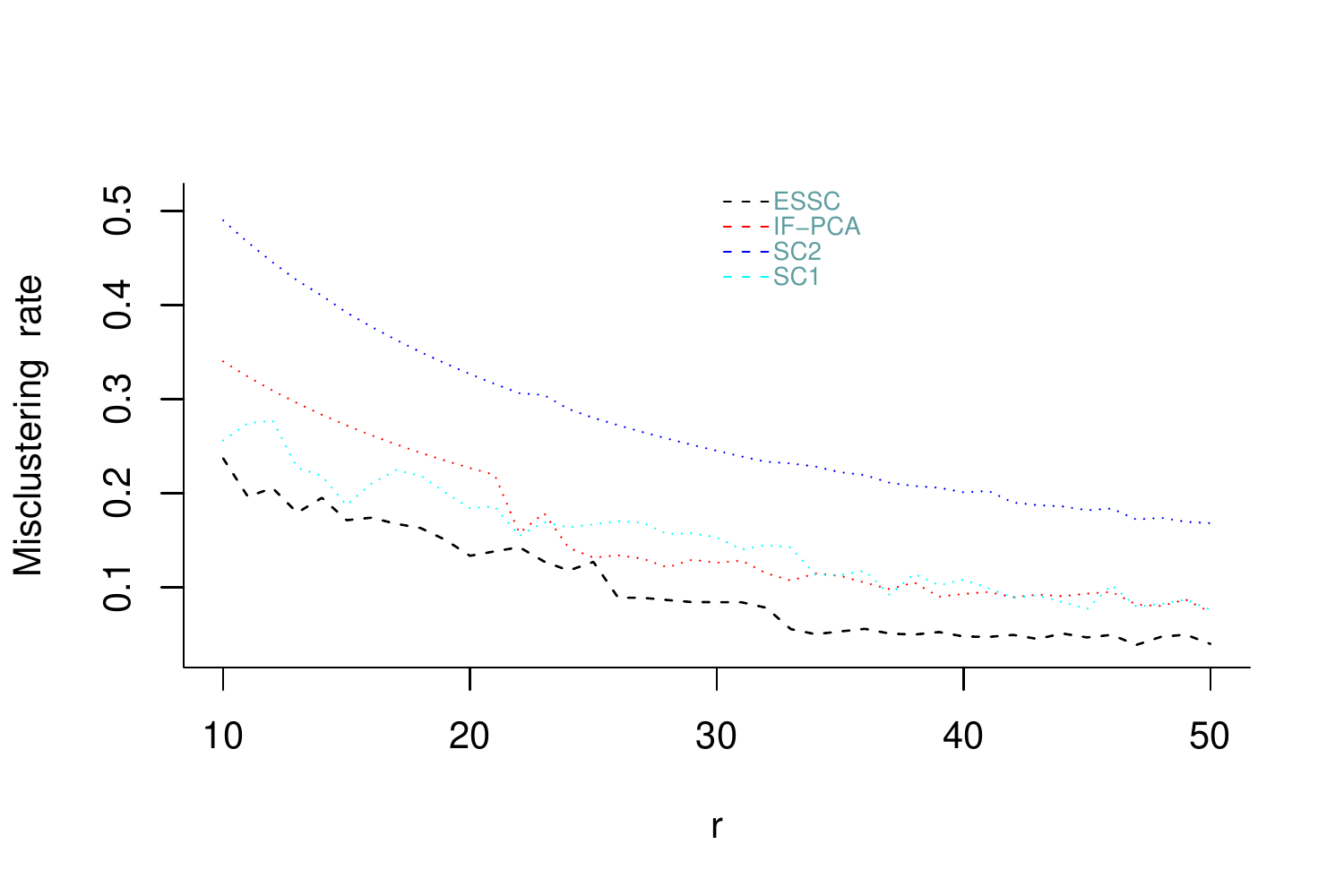}
\caption{Misclustering rate of the Credit card data vs. different sample sizes  $n=5(r+10)$. The red curve represents IF-PCA, the cyan curve represents SC1, the blue curve represents SC2 and the black curve represents ESSC.  }\label{f6}
\end{figure}

\subsubsection{Biological data}
{\color{black} We use several gene microarray data sets collected and processed by authors in \cite{jin2016}.  These data sets are canonical datasets analyzed in the literature such as in \cite{dettling2004}, \cite{gordon2002} and \cite{yousefi2009}.  We use a processed version at www.stat.cmu.edu/~jiashun/Research/software/GenomicsData.  We apply the four approaches mentioned in the financial data section. All the datasets considered in this section belong to the ultra-high-dimensional settings. In each dataset,  the number of features  is about two orders of magnitude larger than the sample size; see Table \ref{tabinf} for a summary. In supervised learning, when feature dimensionality and sample size have such a relation, some independence screening procedure is usually beneficial before implementing methods from joint modeling. We will adopt a similar two-step pipeline for clustering.    As IF-PCA involves an independence screening step via normalized KS-statistic ((1.7) of \cite{jin2016}), we also implement this screening step before calling other methods.  Concretely on each dataset, for each $p\in\{150, 151, 152, \ldots, 300\}$, we keep the $p$ features that have the largest $p$ normalized KS-statistic and construct a $p\times n$ matrix $\bbX$.  Then, since the dimension reduction step is done, for IF-PCA we only apply the ``PCA-2'' step in \cite{jin2016}. Moreover, we subsample each dataset so that the resulting datasets all have an average size of $60$. Concretely, when a dataset has $n$ instances, we keep each instance with a probability $60/n$. For each dataset, we repeat the subsampling procedure $10$ times and report the average misclustering rates of the clustering methods on the subsamples.

\begin{table}[htbp!]
	\centering
	\caption{Sample size and dimensionality of real data sets }
	\begin{tabular}{c|c|c}
		\toprule
	  \text{Data Name}	 & \text{Sample size} & \text{Total number of features}   \\
		\midrule
Colon Cancer        &62&2000\\
Breast Cancer     	  & 276 &22215 \\
Lung Cancer 1       &203&12600\\
Lung Cancer 2     &181&12533\\
Leukemia       &72&3571\\
		\bottomrule
	\end{tabular}\label{tabinf}
\end{table}%

\begin{figure}[hbtp!]
	\centering
\textbf{Misclusering rate of Colon Cancer data}
		\includegraphics[scale=0.42]{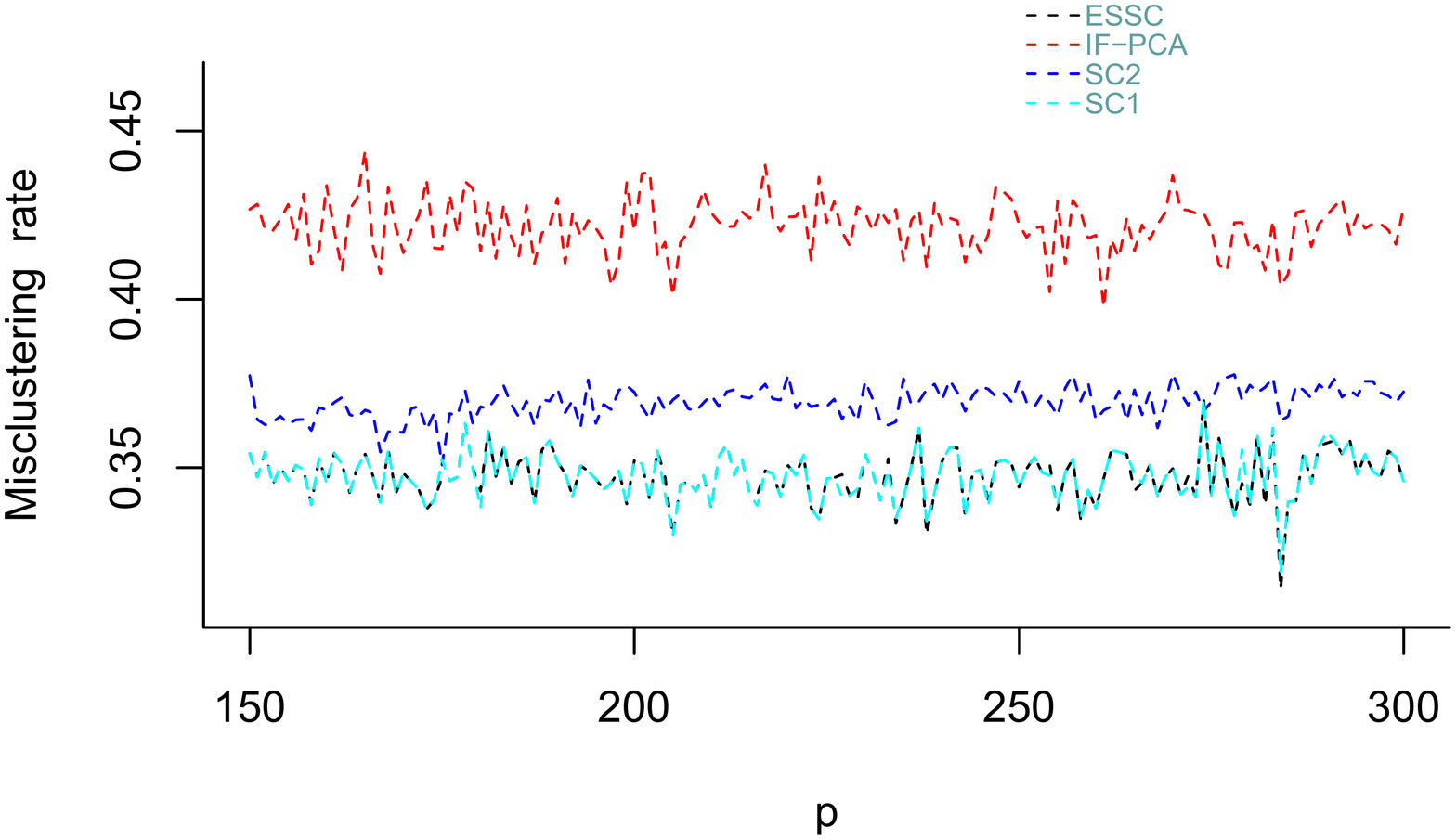}
\caption{Misclustering rate of the Colon Cancer data vs.  different feature dimension  $p$. The red curve represents IF-PCA, the cyan curve represents SC1, the blue curve represents SC2, and the black curve represents ESSC.  }\label{f1}
\end{figure}
\begin{figure}[hbtp!]
\centering
\textbf{Misclusering rate of Breast Cancer data}
		\includegraphics[scale=0.42]{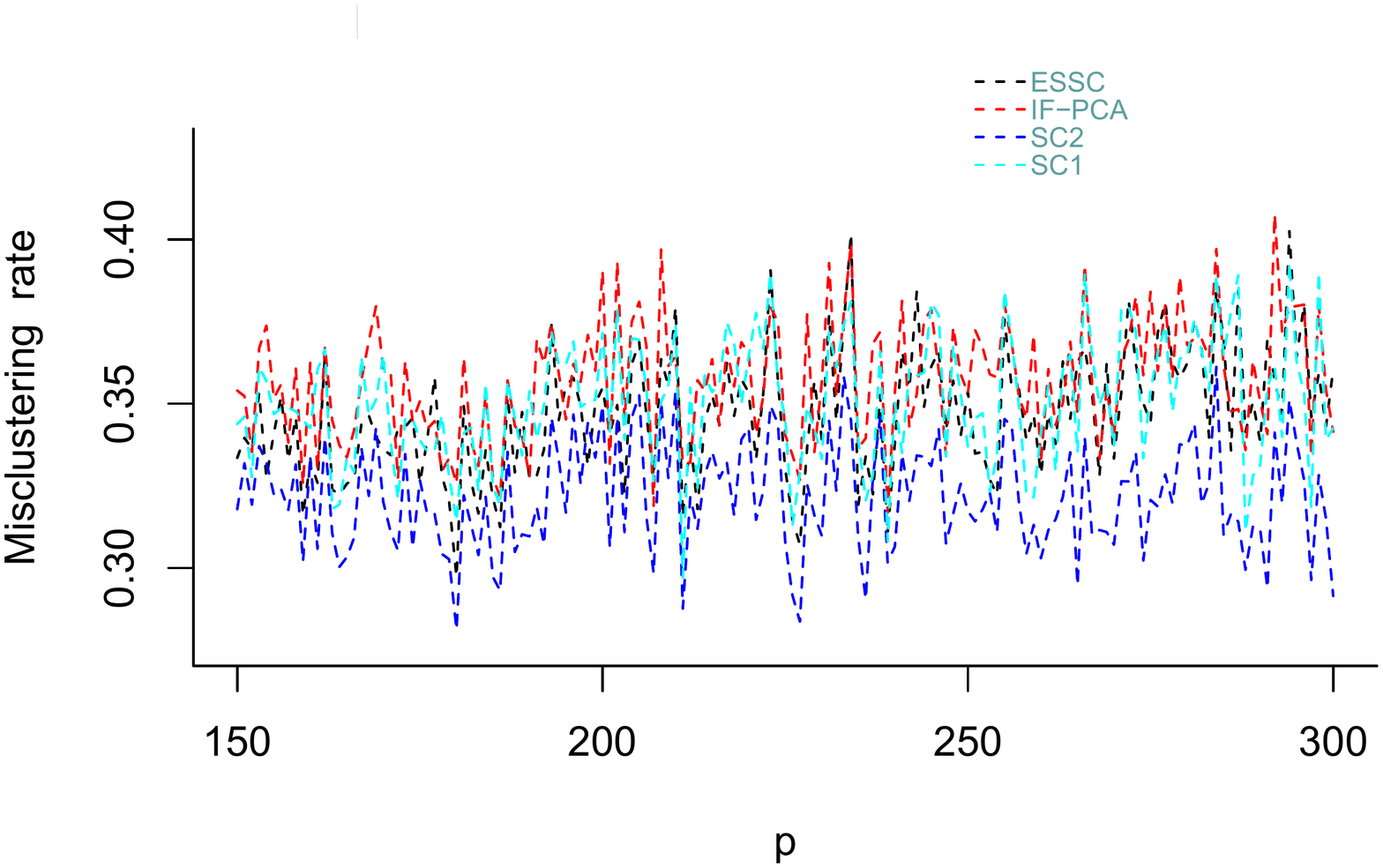}
\caption{Misclustering rate of the Breast Cancer  data vs. different feature dimension  $p$. The red curve represents IF-PCA, the cyan curve represents SC1, the blue curve represents SC2 and the black curve represents ESSC.  }\label{f2}
\end{figure}
\begin{figure}[hbtp!]
\centering
\textbf{Misclusering rate of Lung Cancer 1 data}
		\includegraphics[scale=0.3]{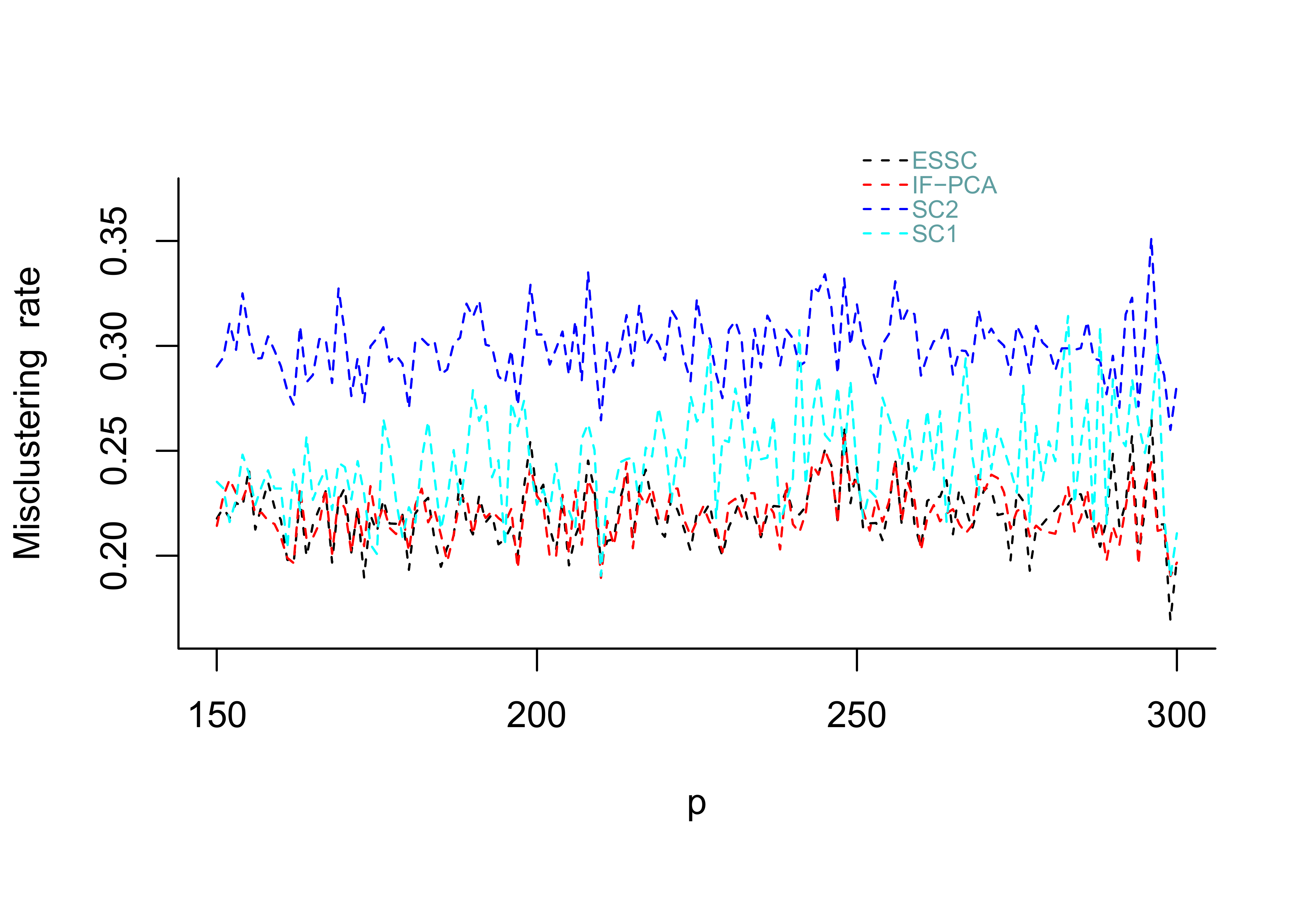}
\caption{Misclustering rate of Lung Cancer 1 data vs. different feature dimension  $p$. The red curve represents IF-PCA, the cyan curve represents SC1, the blue curve represents SC2 and the black curve represents ESSC.  }\label{f3}
\end{figure}
\begin{figure}[hbtp!]
\centering
\textbf{Misclusering rate of Lung Cancer 2 data}
		\includegraphics[scale=0.42]{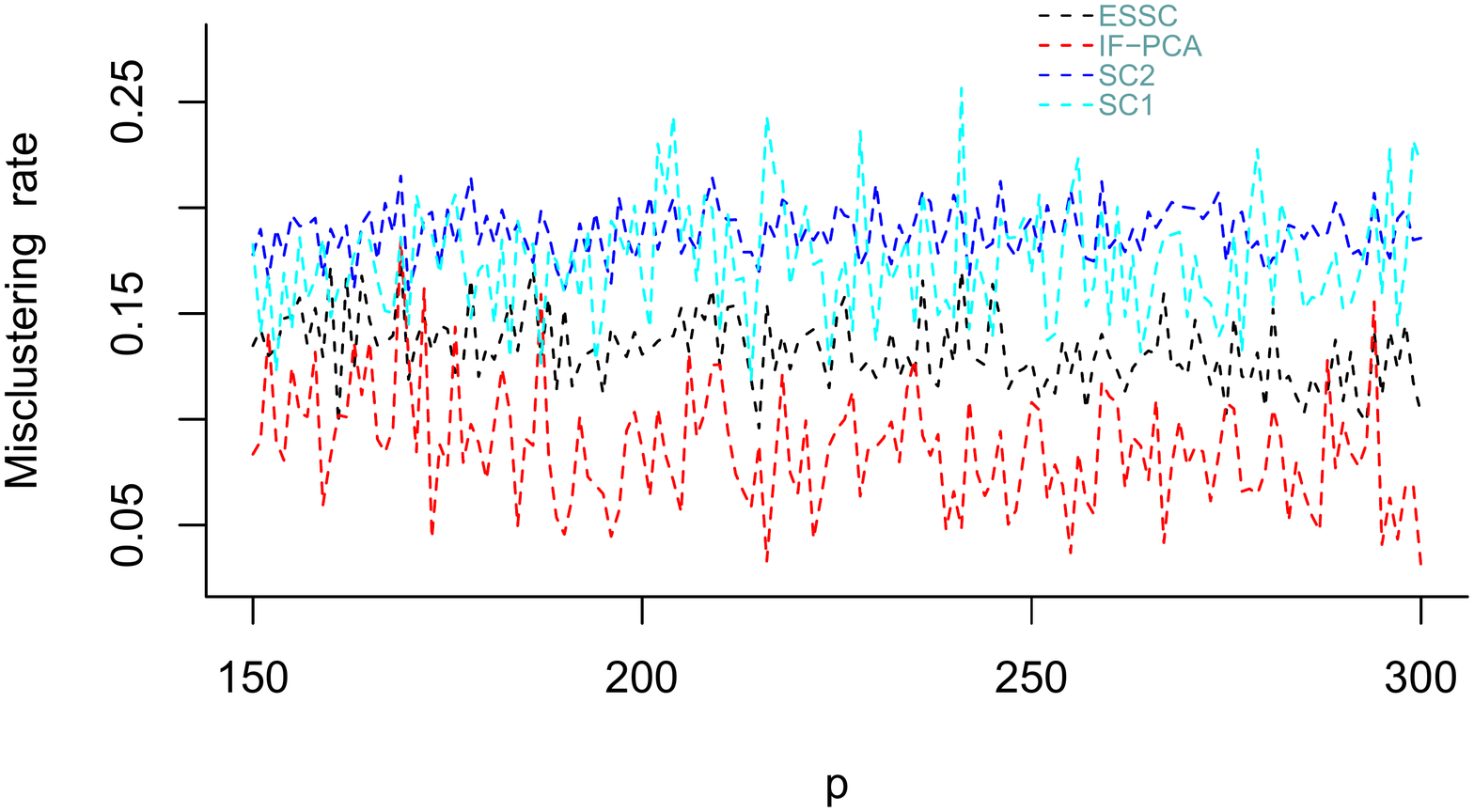}
\caption{Misclustering rate of Lung Cancer 2 data vs. different feature dimension $p$. The red curve represents IF-PCA, the cyan curve represents SC1, the blue curve represents SC2 and the black curve represents ESSC.  }\label{f4}
\end{figure}

\begin{figure}[hbtp!]\label{f5}
\centering
\textbf{Misclusering rate of Leukemia data}
		\includegraphics[scale=0.42]{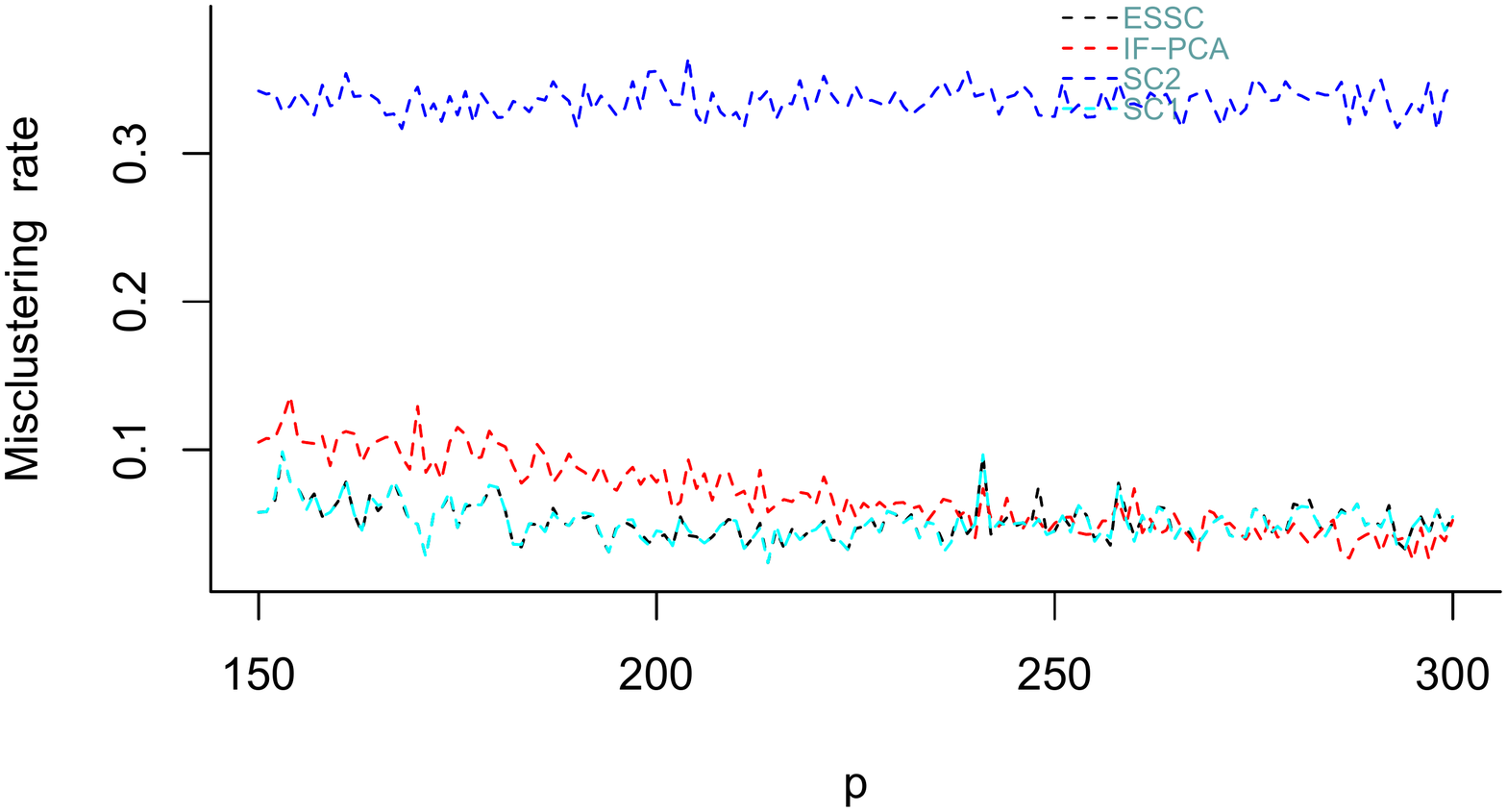}
\caption{Misclustering rate of the Leukemia data vs. different feature dimension  $p$. The red curve represents IF-PCA, the cyan curve represents SC1, the blue curve represents SC2 and the black curve represents ESSC.  }\label{f5}
\end{figure}

From Figures \ref{f1}-\ref{f5}, we compare the methods as follows.   ESSC and SC1 work better than IF-PCA for the  \verb+Colon Cancer+ and  \verb+Leukemia+ data.   For \verb+Lung Cancer 1+ data, ESSC has a similar misclustering rate with IF-PCA in general and outperforms the other two approaches.  For \verb+Breast Cancer+ data, SC2 outperforms the other approaches,  SC1 works a little better than IF-PCA, and ESSC has similar performance with SC1.  For \verb+Lung Cancer 2+ data, IF-PCA has the best performance and ESSC is the second best.  Overall, ESSC belongs to the top two across all five datasets, demonstrating its efficiency and stability.

\section{Discussion}
In this work, with a two-component Gaussian mixture  type model, we propose a theory-backed eigen selection procedure for spectral clustering. The rationale behind the selection procedure is generalizable to more than two components in the mixture. We refer interested readers to  Supplementary Material  for further discussion. Moreover, for future work, it would be interesting to study how an eigen selection procedure might help spectral clustering when a non-linear kernel is used to create an affinity matrix.

\newpage

\renewcommand\thesection{S}
\renewcommand\thefigure{S\arabic{figure}}
\renewcommand\thetable{S\arabic{table}}
\renewcommand{\theequation}{S.\arabic{equation}}
\setcounter{equation}{0}
\setcounter{thm}{3}
\title{\textbf{Supplementary material  to ``Eigen selection in spectral clustering:  a theory guided practice''}}
\subsection{Proof of Theorem \ref{thm: population-eg}}

We use $\bbu=(\bbu(1),\ldots,\bbu(n))^{\top}$ to denote either $\bbu_1$ or $\bbu_2$ and $d^2$ to denote its corresponding eigenvalue, unless specified otherwise.
	
	Because $\bba_1$ only takes two values, by \eqref{0410.1},  there are at most two values of $\bbu(i)$, $i=1,\ldots,n$. We denote these values  by
	$v_1$ and $v_2$.
By (\ref{0409.1}) and (\ref{0409.2}), the number of $v_1$'s in $\bbu$ is either $n_1$ or $n_2$. Without loss of generality, we assume the number of $v_1$'s in $\bbu$ is $n_1$ and the number of $v_2$'s in $\bbu$ is $n_2$.

 Then it follows from (\ref{0409.1}) and  (\ref{0409.2}) that
	\begin{equation}\label{0122.2}
	n_1c_{11}v_1+n_2c_{12}v_2=d^2v_1\,, \text{ and }\quad  n_1c_{12}v_1+n_2c_{22}v_2=d^2v_2\,.
	\end{equation}
	These equations are equivalent to
	\begin{equation}\label{0122.5}
	(d^2-n_1c_{11})v_1=n_2c_{12}v_2\,,
	\end{equation}
	\begin{equation}\label{0122.6}
	n_1c_{12}v_1=(d^2-n_2c_{22})v_2\,.
	\end{equation}

In view of  \eqref{0122.5} and \eqref{0122.6}, we have both $d_1^2$ and $d_2^2$ solve the equation
				\begin{equation}\label{0122.6h}
		(d^2-n_2c_{22})(d^2-n_1c_{11})=n_1n_2c_{12}^2\,.
		\end{equation}
Then \eqref{0122.7h} and \eqref{0122.7ht} follows from \eqref{0122.6h} directly.
 Now let us prove  (a)-(d) of Theorem \ref{thm: population-eg} one by one.
\begin{enumerate}[(a)]
\item When $c_{12}^2 = c_{11}c_{22}$, by \eqref{0122.7h} and \eqref{0122.7ht}  we have $d_1^2 = n_1c_{11}+n_2c_{22}$ and $d_2^2 = 0$. Then $\bbu_2$ does not have clustering power.  Substituting $d_1^2 = n_1c_{11}+n_2c_{22}$ into \eqref{0122.5} and \eqref{0122.6}, we obtain that $\bbu_1 \propto \bone$ if and only if $c_{11}=c_{12}=c_{22}$, which is equivalent to $\mu_1=\mu_2$. This is a contradiction to the condition that $\mu_1\neq \mu_2$ in this paper. Therefore $\bbu_1$ has clustering power.

\item  When $c_{12}=0$, $c_{12}^2 \neq c_{11}c_{22}$ and $n_1c_{11}=n_2c_{22}$, by \eqref{0122.7h} and \eqref{0122.7ht} we conclude that $d_1^2 = d_2^2 = n_1c_{11}$. Since $\bbu_1^\top\bbu_2=0$, it is easy to see that at least one of $\bbu_1$ and $\bbu_2$ has clustering power.

\item When $c_{12}=0$, $c_{12}^2 \neq c_{11}c_{22}$ and $n_1c_{11}\neq n_2c_{22}$, then it follows from  \eqref{0122.7h} and \eqref{0122.7ht} that $d_1^2  = \max\{n_1c_{11},n_2c_{22}\}$ and $d_2^2=\min\{n_1c_{11},n_2c_{22}\}$. Moreover, by $0=c_{12}^2 \neq c_{11}c_{22}$ we have $c_{11},c_{22}>0$, which implies that $d_2^2>0$. Combining these with \eqref{0122.5} and \eqref{0122.6}, we have both $\bbu_1$ and $\bbu_2$ have clustering power. Moreover, both $\bbu_1$ and $\bbu_2$ contain zero entries in view of \eqref{0122.2}.

\item When $c_{12}\neq 0$ and $c_{12}^2 \neq c_{11}c_{22}$. By \eqref{0122.7h} and \eqref{0122.7ht} we have $d_1^2,d_2^2\neq n_1c_{11}\neq 0$, by \eqref{0122.5} we have
		\begin{equation}\label{0122.8}
		v_1=\frac{n_2c_{12}}{d^2-n_1c_{11}}v_2\,.
		\end{equation}
	Therefore if $n_2c_{12}/(d^2-n_1c_{11})\neq 1$, the corresponding eigenvector $\bbu$ has clustering power.  Moreover, in case (d),   $n_2c_{12}/(d^2-n_1c_{11})= 1$ is equivalent to $d^2 = n_1c_{11}+n_2c_{12} = n_1c_{12} + n_2c_{22}$ by \eqref{0122.5} and \eqref{0122.6}. Moreover, the corresponding eigenvector $\bu$ has all entries equal to the same value and thus has no clustering power.
		Since $\bbu_1$ and $\bbu_2$ are orthogonal, when $n_1c_{11}+n_2c_{12} = n_1c_{12} + n_2c_{22}$, exactly one of $\bbu_1$ and $\bbu_2$ has clustering power. If $n_1c_{11}+n_2c_{12} \neq n_1c_{12} + n_2c_{22}$, then $n_2c_{12}/(d_1^2-n_1c_{11})\neq 1$ and $n_2c_{12}/(d_2^2-n_1c_{11})\neq 1$ and thus both $\bbu_1$ and $\bbu_2$ have clustering power.
\end{enumerate}

\subsection{Proof of Proposition \ref{t1}}\label{secta}
The main idea for proving Proposition \ref{t1} is to carefully construct a matrix whose eigenvalue is $\widehat t_k - t_1$, then using similar idea for proving Lemma \ref{0508-1} {\color{black}by analysing the resolvent entries of the matrices such as $(\bbW-z\bbI)^{-1}$}, we can get the desired asymptotic expansions.

By the conditions in Proposition \ref{t1}, 
 for sufficiently large $n$, there exists some positive constant $L$ such that
\begin{equation}\label{0521.1}
\frac{\sigma_n^{L}}{d_1^L}<\frac{1}{2d_1^4}\,,
\end{equation} and in the sequel we fix this $L$. 
 Indeed, $\frac{\sigma_n^L}{d_1^{3L/4}}\ll 1$ and therefore \eqref{0521.1} holds for $L=16$. 

{\color{black}
Assumption (\ref{0513.2}) implies that

\begin{equation}\label{0513.3}
\frac{d_1}{d_2}=1+o(1).
\end{equation}

 It follows from $d_2\gg \sigma_n$  and (\ref{0513.3}) that
\begin{equation}\label{0513.4}
\frac{a_n}{d_2}=1+o(1)\  \text{and}\
\frac{b_n}{d_1}=1+o(1)\,.
\end{equation}
Moreover, it follows from \eqref{0513.3} and Assumption \ref{a2} that
\begin{equation}\label{0513.5}
\frac{\sigma_n}{a_n}\le \frac{1}{2n^{\ep}},\ \text{for some positive constant $\ep$}\,.
\end{equation}}

Throughout the proof, \eqref{0513.5} will be applied in every $O_p(\cdot)$, $o_p(\cdot)$, $O(\cdot)$ and $o(\cdot)$ terms without explicit quotation. We define a Green function of $\bbW$ (defined in \eqref{0506.2}) by
\begin{equation} \label{neweq001}
\bbG(z) = (\bbW-z\bbI)^{-1}, \ z \in \mathbb{C}, \ |z|>\|\bbW\|\,.
\end{equation}

  By  Weyl's inequality,  we have $|\widehat t_k-d_k|\le \|\bbW\|$, $k=1,2$. Thus, {\color{black}by  \eqref{0513.4} and  Lemma \ref{bod}, with probability tending to $1$,}
  \begin{equation}\label{0518.1}
  \min\{\widehat t_2,a_n\}\gg \|\bbW\|\,.
    \end{equation}

Therefore, $\bbG(z)$, $z\in [a_n,b_n]$, $\bbG(\widehat t_1)$ and  $\bbG(\widehat t_2)$ are well defined and nonsingular with probability tending to $1$. Since we only need to show the conclusions of Proposition \ref{t1} hold with probability tending to $1$, in the sequel of this proof, we will assume the existence and nonsingularity of $\bbG(\widehat t_k)$.

By the decomposition of $\E \mathcal{Z}$ in (\ref{0506.1})  and definition of $\bbW$ in  (\ref{0506.2}), we have $\mathcal{Z}=\bbV\bbD\bbV^\top-\bbV_{-}\bbD\bbV^\top_{-}+\bbW$. Then it can be calculated that
\begin{align*}
0 &= \det\left(\mathcal{Z} - \widehat t_k \bbI\right) \\
 &=  \det\left(\bbW-\widehat t_k\bbI+\bbV\bbD\bbV^\top-\bbV_{-}\bbD\bbV^\top_{-}\right) \\
 &=\det\left(\bbG^{-1}(\widehat t_k)+(\bbV\bbD\bbV^\top-\bbV_{-}\bbD\bbV^\top_{-})\right)\non
&=\det\left(\bbG^{-1}(\widehat t_k)\right) \det\left(\bbI+ \bbG(\widehat t_k)(\bbV\bbD\bbV^\top-\bbV_{-}\bbD\bbV^\top_{-})\right)\,, \ k=1,2\,.
\end{align*}
Since $\bbG(\widehat t_k)$ is a nonsingular matrix, $\det[\bbG^{-1}(\widehat t_k)]\neq 0$, which leads to
$$\det\left(\bbI+ \bbG(\widehat t_k)(\bbV\bbD\bbV^\top-\bbV_{-}\bbD\bbV^\top_{-})\right)=0.$$ Notice that
 $(\bbV\bbD\bbV^\top-\bbV_{-}\bbD\bbV^\top_{-})=(\bbV, \bbV_{-})\left(
  \begin{array}{cc}
  \bbD& 0 \\
  0 & -\bbD\\
  \end{array}
\right)(\bbV, \bbV_{-})^\top$. Combining this with the identity $\det(\bbI+\bbA\bbB)=\det(\bbI+\bbB\bbA)$ for any  matrices $\bbA$ and $\bbB$, we have
\begin{equation*} 
0 = \det[\bbI+ \bbG(\widehat t_k)(\bbV\bbD\bbV^\top-\bbV_{-}\bbD\bbV^\top_{-})] = \det\left[\bbI+ \left(
  \begin{array}{cc}
  \bbD& 0 \\
  0 & -\bbD\\
  \end{array}
\right)(\bbV, -\bbV_{-})^\top\bbG(\widehat t_k)(\bbV, -\bbV_{-})\right].
\end{equation*}
 Since $\bbD>0$, it follows from the equation above that
\begin{eqnarray}\label{0314.5h}
\det\left[\left(
  \begin{array}{cc}
  \bbD^{-1}& 0 \\
  0 & -\bbD^{-1}\\
  \end{array}
\right)+ (\bbV, -\bbV_{-})^\top\bbG(\widehat t_k)(\bbV, -\bbV_{-})\right]= 0\,,\text{ } \text{ for } k = 1, 2\,.
\end{eqnarray}

 To analyze \eqref{0314.5h},  we prove some properties of $\bbG(z)$ and the related expressions.
 First of all, by Lemma \ref{0508-1}, we have
\begin{equation}\label{0212.3}
t_k-d_k=O\left(\frac{\sigma_n^2}{a_n}\right), \ k=1,2\,.
\end{equation}
Therefore the distance of $t_k$ and $d_k$ is well controlled and will be used later in this proof.
 Now we turn to analyse $\widehat t_k$, $k=1,2$.  By \eqref{0518.1}, we have
\begin{equation}\label{exp1}
\bbG(z)=(\bbW-z\bbI)^{-1}=-\sum_{i=0}^{\infty}\frac{\bbW^i}{z^{i+1}}\,,
 \end{equation}
 and
\begin{equation}\label{exp2}
\bbG'(z)=-(\bbW-z\bbI)^{-2}=\sum_{i=0}^{\infty}\frac{(i+1)\bbW^i}{z^{i+2}}, \ z\in [a_n,b_n]\,.
  \end{equation}
   By \eqref{0521.1}, (\ref{exp1}), (\ref{exp2}), Lemmas \ref{mos} and \ref{bod}, for any $z\in [a_n,b_n]$ we have

	\begin{align}\label{1011.3h1}
	\bbM_1^\top\bbG(z)\bbM_2&=\bbM_1^\top(\bbW-z\bbI)^{-1}\bbM_2=-\sum_{i=0}^{\infty}\frac{1}{z^{i+1}}\bbM_1^\top\bbW^i\bbM_2\non
	&=\mathcal{R}(\bbM_1,\bbM_2,z)-z^{-2}\bbM_1^\top\bbW\bbM_2-\sum_{i=2}^{L}\frac{1}{z^{i+1}}\bbM_1^\top(\bbW^i-\E\bbW^i)\bbM_2+\tilde \Delta_{n1}\non
&=\mathcal{R}(\bbM_1,\bbM_2,z)-z^{-2}\bbM_1^\top\bbW\bbM_2+\Delta_{n1}\,,
	\end{align}
and
	\begin{align}\label{1011.3h}
	\bbM_1^\top\bbG'(z)\bbM_2&=\bbM_1^\top(\bbW-z\bbI)^{-2}\bbM_2=\sum_{i=0}^{\infty}\frac{i+1}{z^{i+2}}\bbM_1^\top\bbW^i\bbM_2\non
	&=\mathcal{R}'(\bbM_1,\bbM_2,z)+2z^{-3}\bbM_1^\top\bbW\bbM_2+\sum_{i=2}^{L}\frac{i+1}{z^{i+2}}\bbM_1^\top(\bbW^i-\E\bbW^i)\bbM_2+\tilde \Delta_n\non
&=\mathcal{R}'(\bbM_1,\bbM_2,z)+2z^{-3}\bbM_1^\top\bbW\bbM_2+\Delta_n\,,
	\end{align}
where $\|\Delta_{n1}\|=O_p(\frac{\sigma_n}{a_n^3})$, $\|\tilde \Delta_{n1}\|=O_p(\frac{1}{a_n^3})$, $\|\Delta_n\|=O_p(\frac{\sigma_n}{a_n^4})$ and $\|\tilde \Delta_n\|=O_p(\frac{1}{a_n^4})$.
Notice that
$$\mathcal{R}'(\bbM_1,\bbM_2,z)=\frac{\bbM_1^\top\bbM_2}{z^2}+\frac{\bbM_1^\top\E\bbW^2\bbM_2}{z^4}+\sum_{i=3}^{L}\frac{i+1}{z^{i+2}}\bbx^\top\E\bbW^i\bby\,.$$
	It follows from Lemma \ref{mos} and \eqref{0214.2} that for all $z\in [a_n,b_n]$
	\begin{align}\label{1011.4hh}
  \left\|\mathcal{R}(\bbM_1,\bbM_2,z)+z^{-1}\bbM_1^\top\bbM_2\right\|=O(\sigma_n^2/a_n^3)\,,
	\end{align}
and
	\begin{align}\label{1011.4h}
 \left\|\mathcal{R}'(\bbM_1,\bbM_2,z)-z^{-2}\bbM_1^\top\bbM_2\right\|=O(\sigma_n^2/a_n^4)\,.
	\end{align}
	By (\ref{1011.3h1}) and Lemma \ref{mos}, we can conclude that for all $z\in [a_n,b_n]$
	\begin{align}\label{1011.1hh}
	\left\|\bbV^\top\bbG(z)\bbV_{-}\right\|=a_n^{-2}O_p(1)+a_n^{-3}O_p(\sigma^2_n)\,,
	\end{align}
and
	\begin{align}\label{1011.1h}
	\left\|\bbM_1^\top\bbG(z)\bbM_2-\mathcal{R}(\bbM_1,\bbM_2,z)\right\|=\left\|z^{-2}\bbM_1^\top\bbW\bbM_2\right\|+O_p\left(\frac{\sigma_n}{a_n^3}\right)=O_p\left(\frac{1}{a_n^2}\right)\,.
	\end{align}
By  (\ref{1011.4hh}) and (\ref{1011.1h}), we have
	\begin{align}\label{1010.3hh}
	\Big\| & \left(-\bbD^{-1}+\bbV_{-}^\top\bbG(z)\bbV_{-}\right)^{-1}-\left(-\bbD+\mathcal{R}(\bbV_{-},\bbV_{-},z)\right)^{-1}\Big\|\non
	&\le \left\|\bbV_{-}^\top\bbG(z)\bbV_{-}-\mathcal{R}(\bbV_{-},\bbV_{-},z)\right\| \left\|\left(-\bbD^{-1}+\bbV_{-}^\top\bbG(z)\bbV_{-}\right)^{-1}\right\|\left\|\left(-\bbD+\mathcal{R}(\bbV_{-},\bbV_{-},z)\right)^{-1}\right\|\non
	&=O_p(1), \ z\in [a_n,b_n]\,.
	\end{align}
	Moreover, by (\ref{1011.4hh}), (\ref{1011.4h}) and (\ref{1011.1h}) we have
	\begin{align}\label{1010.3h}
	\Big\| & \left[\left(-\bbD^{-1}+\bbV_{-}^\top\bbG(z)\bbV_{-}\right)^{-1}-\left(-\bbD+\mathcal{R}(\bbV_{-},\bbV_{-},z)\right)^{-1}\right]'\Big\|\\
	&=\Big\|\left(-\bbD^{-1}+\bbV_{-}^\top\bbG(z)\bbV_{-}\right)^{-1}\bbV_{-}^\top\bbG'(z)\bbV_{-} \left(-\bbD^{-1}+\bbV_{-}^\top\bbG(z)\bbV_{-}\right)^{-1}\non
	&-\left(-\bbD+\mathcal{R}(\bbV_{-},\bbV_{-},z)\right)^{-1}\mathcal{R}'(\bbV_{-},\bbV_{-},z)\left(-\bbD+\mathcal{R}(\bbV_{-},\bbV_{-},z)\right)^{-1}\Big\|\non
	&=O\left\{\left\|\bbV_{-}^\top\bbG'(z)\bbV_{-}-\mathcal{R}'(\bbV_{-},\bbV_{-},z)\right\| \left\|\left(-\bbD^{-1}+\bbV_{-}^\top\bbG(z)\bbV_{-}\right)^{-1}\right\|^2\right\}\non
	&+O\Big\{\left\|\left[-\bbD^{-1}+\bbV_{-}^\top\bbG(z)\bbV_{-}\right]^{-1}-\left(-\bbD+\mathcal{R}(\bbV_{-},\bbV_{-},z)\right)^{-1}\right\| \nonumber \\
	&
	\quad \cdot \left(\left\|\left(-\bbD^{-1}+\bbV_{-}^\top\bbG(z)\bbV_{-}\right)^{-1}\right\|+\left\|\left(-\bbD^{-1}+\bbV_{-}^\top\bbG(z)\bbV_{-}\right)^{-1}\right\|\right)\|\mathcal{R}'(\bbV_{-},\bbV_{-},z)\|\Big\}\non
	&=O_p\left(\frac{1}{a_n}\right)+O_p\left(\frac{\sigma_n}{a_n^2}\right)\,,\nonumber
	\end{align}
	and
	\begin{align}\label{1010.6h}
	\Big\| & \left\{\left(-\bbD+\mathcal{R}(\bbV_{-},\bbV_{-},z)\right)^{-1}\right\}'\Big\| \\
	&=\left\|\left(-\bbD+\mathcal{R}(\bbV_{-},\bbV_{-},z)\right)^{-1}\mathcal{R}'(\bbV_{-},\bbV_{-},z)\left(-\bbD+\mathcal{R}(\bbV_{-},\bbV_{-},z)\right)^{-1}\right\| \nonumber\\
	&=O(1), \ z\in [a_n,b_n]\,.\nonumber
	\end{align}
	
	By (\ref{1011.3h})--(\ref{1010.3h}), we have the following expansions
\begin{align} \label{0508.9}
\bbV^\top & \bbF(z)\bbV  =\bbV^\top\bbG(z)\bbV_{-}\left(-\bbD^{-1}\bbI+\bbV_{-}^\top\bbG(z)\bbV_{-}\right)^{-1}\bbV_{-}^\top\bbG(z)\bbV \\
&=\mathcal{R}(\bbV,\bbV_{-},z)\left(-\bbD^{-1}\bbI+\mathcal{R}(\bbV_{-},\bbV_{-},z)\right)^{-1}\mathcal{R}(\bbV_{-},\bbV,z)+\Delta_{n2}\,,\nonumber
\end{align}
and
	\begin{align}\label{1011.2h}
	\bbV^\top\bbF'(z)\bbV & =2\bbV^\top\bbG'(z)\bbV_{-}\left(-\bbD^{-1}+\bbV_{-}^\top\bbG(z)\bbV_{-}\right)^{-1}\bbV_{-}^\top\bbG(z)\bbV\\
	&\quad +\bbV^\top\bbG(z)\bbV_{-}\left\{\left(-\bbD^{-1}+\bbV_{-}^\top\bbG(z)\bbV_{-}\right)^{-1}\right\}'\bbV_{-}^\top\bbG(z)\bbV\non
	&= 2\mathcal{R}'(\bbV,\bbV_{-},z)\left(-\bbD+\mathcal{R}(\bbV_{-},\bbV_{-},z)\right)^{-1}\mathcal{R}(\bbV_{-},\bbV,z)\non
	&\quad+\mathcal{R}(\bbV,\bbV_{-},z)\left\{\left(-\bbD+\mathcal{R}(\bbV_{-},\bbV_{-},z)\right)^{-1}\right\}'\mathcal{R}(\bbV_{-},\bbV,z) \nonumber\\
	&\quad+\Delta_{n3}\,,\nonumber
	\end{align}
where $\|\Delta_{n2}\|= O_p(\frac{\sigma_n^2}{a_n^{4}})$ and $\|\Delta_{n3}\|=O_p(\frac{1}{a_n^{4}})+O_p(\frac{\sigma_n^3}{a_n^{6}}).$

Now we turn to \eqref{0314.5h}.
By (\ref{1011.3h1}), (\ref{1011.4hh}) and (\ref{1011.1h}), we can see that
$\|\bbV^\top\bbG(\widehat t_k)\bbV_{-}\| = O_p(\frac{1}{a_n^2})$, $|\bbv_1\bbG(\widehat t_k)\bbv_2| = O_p(\frac{1}{a_n^2})$ and $|\bbv_{-1}\bbG(\widehat t_k)\bbv_{-2}| = O_p(\frac{1}{a_n^2})$. In other words, the off diagonal terms in the determinant (\ref{0314.5h}) are all $O_p(\frac{1}{a_n^2})$.

 The $3$rd diagonal entry in the determinant (\ref{0314.5h}) is $\bbv_{-1}^\top\bbG(\widehat t_k)\bbv_{-1}-\frac{1}{d_1}$. By  (\ref{1011.3h1}), (\ref{1011.4hh}) and (\ref{1011.1h}), we have $\bbv_{-1}^\top\bbG(\widehat t_k)\bbv_{-1}=-\frac{1}{d_k}+o_p(\frac{1}{a_n})$. i.e. $\bbv_{-1}^\top\bbG(\widehat t_k)\bbv_{-1}-\frac{1}{d_1}=-\frac{1}{d_k}-\frac{1}{d_1}+o_p(\frac{1}{a_n})$. Similarly, the $4$th diagonal entry is $\bbv_{-2}^\top\bbG(\widehat t_k)\bbv_{-2}-\frac{1}{d_2}=-\frac{1}{d_k}-\frac{1}{d_2}+o_p(\frac{1}{a_n})$.   Therefore the matrix $\bbV_{-}^\top\bbG(\widehat t_k)\bbV_{-}-\bbD^{-1}$ is invertible with  probability tending to 1. Recalling the determinant formula for block structure matrix that
$$\det\left(
\begin{array}{ccc}
\bbA &\bbB^\top\\
\bbB & \bbC\\
\end{array}
\right)=\det(\bbC)\det(\bbA-\bbB^\top\bbC^{-1}\bbB)\,,$$
for any invertible matrix $\bbC$ and setting $\bbC=\bbV_{-}^\top\bbG(\widehat t_k)\bbV_{-}-\bbD$,  we have with  probability tending to 1,
\begin{equation}\label{0211.1}
\det(\bbV^\top(\bbG(\widehat t_k)-\bbF(\widehat t_k))\bbV+\bbD^{-1})=0\,,
 \end{equation}
 where $\bbF(z)=\bbG(z)\bbV_{-}\left(-\bbD^{-1}+\bbV_{-}^\top\bbG(z)\bbV_{-}\right)^{-1}\bbV_{-}^\top \bbG(z)$.

	The three equations \eqref{1011.3h}, \eqref{1011.4h} and \eqref{1011.2h}  lead to
\begin{align} \label{0515.2c}
\|\bbV^\top\left(\bbG'(z)-\bbF'(z)\right)\bbV-\frac{1}{z^2}\mathcal{\widetilde P}_{z}^{-1}-2z^{-3}\bbV^\top\bbW\bbV\|=O_p\left(\frac{\sigma_n}{a_n^4}\right)\,,
\end{align}
for $z\in [a_n, b_n]$, where
$$\mathcal{\widetilde P}_{z}^{-1}=z^2\left(\frac{A_{\bbV,z}}{z}\right)'\,,$$
and
\begin{equation}\label{0619.2}
A_{\bbV,z}=\left\{t\mathcal{R}(\bbV,\bbV,z)-z\mathcal{R}(\bbV,\bbV_{-},z)\left(-\bbD+\mathcal{R}(\bbV_{-},\bbV_{-},z)\right)^{-1}\mathcal{R}(\bbV_{-},\bbV,z)\right\}^\top\,.
\end{equation}

Further, recalling the definition in (\ref{0619.2}), it holds that
\begin{align}\label{1011.5h}
\frac{1}{z^2}\mathcal{\widetilde P}_{z}^{-1}&=\left(\frac{A_{\bbV,z}}{z}\right)'=\mathcal{R}'(\bbV,\bbV,z)- 2\mathcal{R}'(\bbV,\bbV_{-},z)\left(-\bbD+\mathcal{R}(\bbV_{-},\bbV_{-},z)\right)^{-1} \nonumber\\
&\quad \times \mathcal{R}(\bbV_{-},\bbV,z) -\mathcal{R}(\bbV,\bbV_{-},z)\left\{\left(-\bbD+\mathcal{R}(\bbV_{-},\bbV_{-},z)\right)^{-1}\right\}'\mathcal{R}(\bbV_{-},\bbV,z)\,.
\end{align}
By (\ref{1011.4hh}), (\ref{1011.4h}) and (\ref{1010.6h}), we have
$$\|\mathcal{\widetilde P}_{z}^{-1}-\bbI\|=O\left(\frac{\sigma_n^2}{a_n^2}\right)\,.$$
Plugging this into \eqref{0515.2c} and by Lemmas \ref{mos}, we have for all $z\in [a_n, b_n]$,
\begin{align}\label{0515.2b1}
\| \bbV^\top \left(\bbG'(z)-\bbF'(z)\right)\bbV-z^{-2}\bbI-2z^{-3}\bbV^\top\bbW\bbV\|=a_n^{-4}O_p(\sigma_n^2)\,.
\end{align}
Hence there exists a $2\times 2$ random matrix $\bbB$ such that
\begin{equation}\label{0515.2b}
  \bbV^\top \left(\bbG'(z)-\bbF'(z)\right)\bbV=z^{-2}\bbB(z),
\end{equation}
where $\|\bbB(z)-\bbI\|=O_p(a_n^{-1} + a_n^{-2}\sigma_n^2).$

Further, in light of expressions (\ref{1011.3h1}) and \eqref{0508.9}, we can obtain the asymptotic expansion
\begin{align} \label{0515.2}
\|\bbI & +\bbD\bbV^\top\left(\bbG(z)-\bbF(z)\right)\bbV- f(z) + z^{-2}\bbD\bbV^\top\bbW\bbV\|= O_p(a_n^{-2}\sigma_n)\,,
\end{align}
for all $z\in [a_n, b_n]$, where $f(z)$ is defined in (\ref{eq2}).

	In view of (\ref{0515.2}) and the definition of $t_k$, we have
	\begin{equation} \label{neweq019}
	\left\|\bbI+\bbD\bbV^\top\left(\bbG(t_k)-\bbF(t_k)\right)\bbV-f(t_k)+t^{-2}_k\bbD\bbV^\top\bbW\bbV\right\|=O_p\left(\frac{\sigma_n}{a_n^2}\right)\,, \ k=1,2\,.
	\end{equation}
	By \eqref{0211.1}, (\ref{0515.2b}) and (\ref{neweq019}), an application of the mean value theorem yields
		\begin{align} \label{neweq018}
		0 & = \det(\bbI +\bbD\bbV^\top\left(\bbG(\widehat t_k)-\bbF(\widehat t_k)\right)\bbV)=\det(\bbI+\bbD\bbV^\top\left(\bbG(t_1)-\bbF(t_1)\right)\bbV \nonumber\\
		&\quad +\bbD\tilde\bbB(\widehat t_k-t_1))\,, \ k=1,2\,,
		\end{align}
		where $\tilde\bbB=(\tilde B_{ij}(\tilde t_{ij}))$, $\tilde t_{ij}^2\tilde B_{ij}(\tilde t_{ij})=\delta_{ij}+O_p(a_n^{-1} + a_n^{-2}\sigma_n^2)$ by (\ref{0515.2b}) and $\widetilde{t}_{ij}$ is some number between $t_1$ and $\widehat t_k$.
By (\ref{0515.2}), similar to (\ref{0211.3})--(\ref{0518.6}), we can show that
\begin{equation}\label{c1}
|\widehat t_k-t_1|=O_p\left(1+\frac{\sigma_n^2}{a_n}\right)+|d_1-d_k|\,.
\end{equation}
 (\ref{neweq018}) can be rewritten as
		\begin{align} \label{neweq018h}
		0 & = \det(\bbI +\bbD\bbV^\top\left(\bbG(\widehat t_k)-\bbF(\widehat t_k)\right)\bbV)=\det(\bbI+\bbD\bbV^\top\left(\bbG(t_1)-\bbF(t_1)\right)\bbV \nonumber\\
		&\quad + t^{-2}_1\bbD\bbC(\widehat t_k-t_1))\,,\  k=1,2\,,
		\end{align}
where
\begin{equation}\label{0224.1}
\|\bbC-\bbI\|=O_p\left(a_n^{-1} +a_n^{-2}\sigma_n^2+\frac{d_1-d_2}{a_n}\right)\,.
\end{equation}
We know that $\widehat t_k-t_1$, $k=1,2$ are the eigenvalues  of $t_1^2\bbC^{-1}\bbD^{-1}\left(\bbI+\bbD\bbV^\top\left(\bbG(t_1)-\bbF(t_1)\right)\bbV\right)$.  Combining (\ref{0212.3}) with the definition of $g(z)$ in (\ref{0227.1}),  we have $g_{ij}(t_k)=O(\frac{\sigma_n^2}{a_n}+d_1-d_2)+O_p(1)$, $1\le i,j,k\le 2$.  The asymptotic expansions in (\ref{neweq019}), (\ref{0224.1}) and Lemma \ref{0505-1} together with the condition
\eqref{0513.2} and \eqref{0513.4} imply that
\begin{align}\label{0507.1}
t_1^2\bbC^{-1}\bbD^{-1}\left(\bbI+\bbD\bbV^\top\left(\bbG(t_1)-\bbF(t_1)\right)\bbV\right)=g(t_1)+\Delta_{n4}\,,
 \end{align}
 where $\Delta_{n4}$ is a symmetric matrix with $\|\Delta_{n4}\|=o_p(1)$. By \eqref{0507.1}, we can rewrite \eqref{neweq018h} as follows,
 		\begin{align} \label{0515.1}
		\det(g(t_1)+\Delta_{n4} +(\widehat t_k-t_1)\bbI)=0,\  k=1,2\,.
		\end{align}
Moreover, by \eqref{0227.1}, the eigenvalues of $g(t_1)$ are
	\begin{align} \label{0515.2h}
		\frac{1}{2}\left[-g_{11}(t_1)-g_{22}(t_1)\pm\left\{\left(g_{11}(t_1)+g_{22}(t_1)\right)^2-4\left(g_{11}(t_1)g_{22}(t_1)-g^2_{12}(t_1)\right)\right\}^{\frac{1}{2}}\right]\,.
		\end{align}
 Combining (\ref{0515.1})--(\ref{0515.2h}) with Weyl's inequality and noticing that $\widehat t_1>\widehat t_2$, we have the following expansions
	\begin{equation} \label{0516.15}
	\widehat t_1-t_1=\frac{1}{2}\left[-g_{11}(t_1)-g_{22}(t_1)+\left\{\left(g_{11}(t_1)+g_{22}(t_1)\right)^2-4\left(g_{11}(t_1)g_{22}(t_1)-g^2_{12}(t_1)\right)\right\}^{\frac{1}{2}}\right]+o_{p}(1)\,,
	\end{equation}
and
	\begin{equation}
	\widehat t_2-t_1=\frac{1}{2}\left[-g_{11}(t_1)-g_{22}(t_1)-\left\{\left(g_{11}(t_1)+g_{22}(t_1)\right)^2-4\left(g_{11}(t_1)g_{22}(t_1)-g^2_{12}(t_1)\right)\right\}^{\frac{1}{2}}\right]+o_{p}(1).
	\end{equation}
Expanding the determinant at $t_2$ in \eqref{neweq018} and repeating the process from \eqref{neweq018}--\eqref{0515.2}, we also have
	\begin{equation} \label{0516.15}
	\widehat t_2-t_2=\frac{1}{2}\left[-g_{11}(t_2)-g_{22}(t_2)-\left\{\left(g_{11}(t_2)+g_{22}(t_2)\right)^2-4\left(g_{11}(t_2)g_{22}(t_2)-g^2_{12}(t_2)\right)\right\}^{\frac{1}{2}}\right]+o_{p}(1).
	\end{equation}

\subsection{More discussion of Proposition \ref{t1}}
\setcounter{lem}{1}
In this section we show that the major terms at the right hand sides of \eqref{th1} and \eqref{th2} are meaningful, as shown in the following lemma.
\begin{lem}\label{0625-1}
\begin{equation}\label{0510.1}
\frac{1}{2}\left[-g_{11}(t_1)-g_{22}(t_1)+\left\{\left(g_{11}(t_1)+g_{22}(t_1)\right)^2-4\left(g_{11}(t_1)g_{22}(t_1)-g^2_{12}(t_1)\right)\right\}^{\frac{1}{2}}\right]=O_p(1)\,,
 \end{equation}
 and
 \begin{equation}\label{0510.2}
\frac{1}{2}\left[-g_{11}(t_2)-g_{22}(t_2)-\left\{\left(g_{11}(t_2)+g_{22}(t_2)\right)^2-4\left(g_{11}(t_1)g_{22}(t_2)-g^2_{12}(t_2)\right)\right\}^{\frac{1}{2}}\right]=O_p(1)\,.
 \end{equation}
 \end{lem}
 \begin{proof}
The proofs of \eqref{0510.1} and \eqref{0510.2} are the same, so we only prove \eqref{0510.1}.

By Lemma \ref{mos}, we have $g_{ij}(t_1)=\frac{t_1^2}{d_i}f_{ij}(t_1)+O_p(1)$. Therefore it suffices to show that
$$\frac{1}{2}\left[-\frac{t_1^2}{d_i}f_{11}(t_1)-\frac{t_1^2}{d_2}f_{22}(t_1)+\left\{\left(g_{11}(t_1)+g_{22}(t_1)\right)^2-4\left(g_{11}(t_1)g_{22}(t_1)-g^2_{12}(t_1)\right)\right\}^{\frac{1}{2}}\right]=O_p(1)\,.$$
{\color{black} By Lemma \ref{mos}}, for any $\ep>0$, there exists a constant $M_0$ such that
$$\p\left(\|\bbV^\top\bbW\bbV\|\ge M_0\right)\le \ep\,.$$
Now we consider the inequality constraint on the event $\{\|\bbV^\top\bbW\bbV\|\le M_0\}$. Let $h_1=\frac{t_1^2}{d_1}f_{11}(t_1)+\frac{t_1^2}{d_2}f_{22}(t_1)$.  It follows from the definition of $t_1$, (\ref{0211.2}){\color{blue}, \eqref{0504.1} and} \eqref{0211.3} that
$$f_{11}(t_1)\ge 0\,, \ \text{and} \ f_{22}(t_1)\ge 0\,.$$
{\color{black}Let $$h_2=2h_1(\bbv_1^\top\bbW\bbv_1+\bbv_2^\top\bbW\bbv_2)-4\frac{t_1^2}{d_1}f_{11}(t_1)\bbv_2^\top\bbW\bbv_2-4\frac{t_1^2}{d_2}f_{22}(t_1)\bbv_1^\top\bbW\bbv_1+4t_1^2\left(\frac{f_{12}(t_1)}{d_1}+\frac{f_{21}(t_1)}{d_2}\right)\bbv_1^\top\bbW\bbv_2\,,$$ and
$$h_3=(\bbv_1^\top\bbW\bbv_1-\bbv_2^\top\bbW\bbv_2)^2+4(\bbv_1^\top\bbW\bbv_2)^2\,.$$}
By the definition of $g$ and the above equations, we have
$$(g_{11}(t_1)+g_{22}(t_1))^2-4\left(g_{11}(t_1)g_{22}(t_1)-g^2_{12}(t_1)\right)=h_1^2+h_2+h_3\,.$$
Note that $|h_2|\le M_1|h_1|$ and $|h_3|\le M_2$, where $M_1$ and $M_2$ are polynomial functions of $M_0$ (depending on $M_0$ only). Now we consider two cases:

1. $|h_3|\le |h_1|$, then we have $|h_2+h_3|\le (M_2+1)|h_1|$. Then
\begin{eqnarray*}
&&\left|-\frac{t_1^2}{d_1}f_{11}(t_1)-\frac{t_1^2}{d_2}f_{22}(t_1)+\left\{\left(g_{11}(t_1)+g_{22}(t_1)\right)^2-4\left(g_{11}(t_1)g_{22}(t_1)-g^2_{12}(t_1)\right)\right\}^{\frac{1}{2}}\right|\non
&&=|-h_1+(h_1^2+h_2+h_3)^{\frac{1}{2}}|= \frac{|h_2+h_3|}{h_1+\left(h_1^2+h_2+h_3\right)^{\frac{1}{2}}}\le M_2+1\,.
\end{eqnarray*}

2. $|h_3|\ge |h_1|$, then
\begin{eqnarray}
&&\left|-\frac{t_1^2}{d_1}f_{11}(t_1)-\frac{t_1^2}{d_2}f_{22}(t_1)+\left\{\left(g_{11}(t_1)+g_{22}(t_1)\right)^2-4\left(g_{11}(t_1)g_{22}(t_1)-g^2_{12}(t_1)\right)\right\}^{\frac{1}{2}}\right|\\
&&=|-h_1+(h_1^2+h_2+h_3)^{\frac{1}{2}}|\le (M_2+1)^2+M_1M_2\,.\nonumber
\end{eqnarray}
Combining the two cases, we have shown that given $\|\bbV^\top\bbW\bbV\|\le M_0$, there exists $M_3$ depending on $M_0$ only such that
$$\left|\frac{1}{2}\left[-g_{11}(t_1)-g_{22}(t_1)+\left\{\left(g_{11}(t_1)+g_{22}(t_1)\right)^2-4\left(g_{11}(t_1)g_{22}(t_1)-g^2_{12}(t_1)\right)\right\}^{\frac{1}{2}}\right]\right|\le M_3\,.$$
In other words,
$$\frac{1}{2}\left[-g_{11}(t_1)-g_{22}(t_1)+\left\{\left(g_{11}(t_1)+g_{22}(t_1)\right)^2-4\left(g_{11}(t_1)g_{22}(t_1)-g^2_{12}(t_1)\right)\right\}^{\frac{1}{2}}\right]=O_p(1)\,.$$
This concludes the proof of Lemma \ref{0625-1}.
\end{proof}

\subsection{Proof of Theorem \ref{t2}}
By Lemma \ref{bod} and weyl's inequality $|\widehat t_k-d_k|\le \|\bbW\|$, $k=1,2$, we have
$$\p\left(\widehat t_2\ge d_2-C_0\max\{n^{\frac{1}{2}}, p^{\frac{1}{2}}\}\right)\ge 1-n^{-2}\,,$$
and
$$\p\left(\widehat t_1\le d_1+C_0\max\{n^{\frac{1}{2}}, p^{\frac{1}{2}}\}\right)\ge 1-n^{-2}\,,$$
for some positive constant $C_0$ and sufficiently large $n$.
Combining the above two equations with $d_1\gg \sigma_n$, and $d_1 / d_2\le 1+n^{-c}$, we have
$$\p\left(\frac{\widehat t_1}{\widehat t_2}\ge 1+C\left(\frac{\sigma_n}{d_1}+\frac{1}{n^{c}}\right)\right)\rightarrow 0\,,$$
where $C$ is some positive constant.

\subsection{Proof of Theorem \ref{t3}}
By Lemma \ref{bod}, there exists a constant $C>0$ such that
\begin{equation}\label{0605.5}
\p\left(\|\bbW\|\ge C\max\{n^{\frac{1}{2}}, p^{\frac{1}{2}}\}\right)\le n^{-D}.
\end{equation}

By Weyl's inequality, we have
\begin{equation}\label{0329.3c}
\max_{i=1,2}|\widehat t_i-d_i|\le \|\bbW\|\,.
\end{equation}
By \eqref{0329.3c} and the condition that $d_1\ge (1+c)d_2$, we have
\begin{equation}\label{0329.1c}
\frac{\widehat t_1}{\widehat t_2}\ge \frac{d_1-\|\bbW\|}{d_2+\|\bbW\|}\ge \frac{1+c-\frac{\|\bbW\|}{d_2}}{1+\frac{\|\bbW\|}{d_2}}\,.
\end{equation}
If $d_2\ge \frac{c}{c+4}C\max\{n^{\frac{1}{2}}, p^{\frac{1}{2}}\}$, by \eqref{0605.5} and \eqref{0329.1c}, we have
$$\p\left(\frac{\widehat t_1}{\widehat t_2}\le 1+\frac{c}{2}\right)\le \p\left(\frac{1+c-\frac{\|\bbW\|}{d_2}}{1+\frac{\|\bbW\|}{d_2}}\le 1+\frac{c}{2}\right)\le n^{-D}$$
If $d_2< \frac{c}{c+4}C\max\{n^{\frac{1}{2}}, p^{\frac{1}{2}}\}$, by the condition that $d_1\gg \sigma_n$, \eqref{0605.5} and \eqref{0329.1c}, for sufficiently large $n$  we have
\begin{equation}\label{0329.2c}
\p\left(\frac{\widehat t_1}{\widehat t_2}\le 1+\frac{c}{2}\right)\le n^{-D}\,.
\end{equation}
This together with the assumption that $d_1 / d_2 \ge 1+c$  implies \eqref{0605.1}. Now we turn to \eqref{0605.2}. Let $\widehat\bbu_1=(\widehat\bbv_1(1),\ldots,\widehat\bbv_1(n))^\top$ and $\widehat{\mathfrak{u}}_1=(\widehat\bbv_1(n+1),\ldots,\widehat\bbv_1(n+p))^\top$. Notice that $\widehat\bbv_1$ is the unit eigenvector of $\mathcal{Z}$ corresponding to $\widehat{d}_1$. By the definition of $\mathcal{Z}$, we know that $2^{1/2}\widehat\bbu_1$ is the unit eigenvector of $\bbX^\top\bbX$ corresponding to $\widehat{d}_1^2$ and $2^{1/2}\widehat{\mathfrak{u}}_1$ is the unit eigenvector of $\bbX\bbX^\top$ corresponding to $\widehat{d}_1^2$.  Similarly, by the condition that the first $n$ entries of $\bbv_1$ are equal, we imply that the first entries of $\bbv_1$ are equal to $(2n)^{-1/2}$. Let $\mathbf{1}_n$ be an $n$-dimensional vector whose entries are all $1$'s.
 By the second inequality of Theorem 10 in the supplement of \cite{cai2013},  we obtain that
\begin{equation}\label{0605.3}
2-2(\bbv_1^\top\widehat\bbv_1)^2\le \frac{\|\bbW\|}{d_1-d_2-\|\bbW\|}\,.
\end{equation}
Since  $d_1/d_2 \geq 1+c$, we have
\begin{equation}\label{0605.4}
d_1-d_2   \ge c(1+c)^{-1}d_1\,.
\end{equation}
Let $C_0=\max\{c(1+c)^{-1},C\}-1$, where $C$ is given in \eqref{0605.5}. By \eqref{0605.5}, \eqref{0605.3} and \eqref{0605.4}, we imply that
$$\p\Big(2-2(\bbv_1^\top\widehat\bbv_1)^2\le \frac{(C_0+1)(\frac{\sigma_n}{d_1})^{2/3}}{C_0}\Big)\ge 1-n^{-D}\,.$$
\begin{equation}\label{0605.7}
\p\Big(|\bbv_1^\top\widehat\bbv_1|\ge 1-\sqrt{\frac{2\sigma_n}{d_1}} \Big)\ge 1-n^{-D}\,,
\end{equation}
where $n\ge n_0(\ep,D)$. Notice that $2^{\frac{1}{2}}\mathfrak{\widehat{u}_1}$ is a unit vector, we have
 $$|\bbv_1^\top\widehat\bbv_1|\le |\mathbf{1}_n^\top\widehat\bbu_1|+\frac{1}{ 2}=\frac{1}{(2n)^{\frac{1}{2}}}|\bbu_0^\top\widehat\bbv_1|+\frac{1}{2}\,.$$
 This together with \eqref{0605.7} implies that
\begin{equation}
\p\left(
\left|\left(\frac{1}{n}\right)^{\frac{1}{2}}|\bbu_0^\top\widehat\bbv_1|-\left(\frac{1}{2}\right)^{\frac{1}{2}}\right|\ge\sqrt{\frac{2\sigma_n}{d_1}} \right)\le n^{-D}\,.
\end{equation}
This completes the proof.

\subsection{Technical Lemmas and their proofs}
\begin{lem}\label{mos}

{\color{black} Take (i)  in Assumption \ref{a2}. For $\bbX$ we considered in this paper and any positive integer $l$, there exists a positive constant $C_l$ (depending on $l$) such that }
\begin{equation}\label{th6}
\E|\bbx^{\top}(\bbW^l-\E\bbW^l)\bby|^2\le C_l\sigma_n^{l-1}\,,
\end{equation}
and $\E\bbx^{\top}\bbW\bby =0$ and
\begin{equation}\label{th7}
|\E\bbx^{\top}\bbW^l\bby|\le C_l\sigma_n^{l}\,, \text{ for } l\geq 2.
\end{equation}
where $\bbx$ and $\bby$ are two unit vectors (random or not random) independent of $\bbW$.
\end{lem}
\begin{proof}
Let $\mathcal{Y}=\Sigma^{-\frac{1}{2}}(\bbX-\E\bbX)$. Recall that $\bbX=(X_1,\ldots,X_n)$ is defined in  \eqref{0830.1h} by
$$X_i=Y_i\bmu_1+(1-Y_i)\bmu_2+W_i,\ i=1,\ldots,n\,,$$
where $\{W_i\}_{i=1}^{n}$ are i.i.d. from $\mathcal{N}(0, \Sigma)$. The entries of $\mathcal{Y}$ are i.i.d. standard normal random variables. Moreover, we decompose  $\bbW$ defined in \eqref{0506.2} by
$$\bbW=\left(
\begin{array}{cc}
\bbI & 0 \\
0 & \Sigma^{\frac{1}{2}} \\
\end{array}
\right)\left(
\begin{array}{cc}
0 & \mathcal{Y}^{\top} \\
\mathcal{Y} & 0 \\
\end{array}
\right)\left(
\begin{array}{cc}
\bbI & 0 \\
0 & \Sigma^{\frac{1}{2}} \\
\end{array}
\right)\,.$$
Let  the eigen decomposition of $\Sigma$ be $\bbU\Lambda\bbU^{\top}$. Since the entries of $\mathcal{Y}$ are i.i.d. standard normal random variables, we have $\mathcal{Y}\stackrel{d}{=}\bbU\mathcal{Y}$. Then $\bbW$ can be written as
$$\bbW\stackrel{d}{=}\left(
\begin{array}{cc}
\bbI & 0 \\
0 & \bbU \\
\end{array}
\right)\left(
\begin{array}{cc}
0 & \mathcal{Y}^{\top}\Lambda \\
\Lambda\mathcal{Y} & 0 \\
\end{array}
\right)\left(
\begin{array}{cc}
\bbI & 0 \\
0 & \bbU^{\top} \\
\end{array}
\right)\,.$$
Therefore
$$\bbx^{\top}\bbW^l\bby=\bbx^{\top}\left(
\begin{array}{cc}
\bbI & 0 \\
0 & \bbU \\
\end{array}
\right)\left(
\begin{array}{cc}
0 & \mathcal{Y}^{\top}\Lambda \\
\Lambda\mathcal{Y} & 0 \\
\end{array}
\right)^L\left(
\begin{array}{cc}
\bbI & 0 \\
0 & \bbU^{\top} \\
\end{array}
\right)\bby\,.$$
Let $\widetilde \bbx=\left(
\begin{array}{cc}
\bbI & 0 \\
0 & \bbU^{\top} \\
\end{array}
\right)\bbx$, $\widetilde \bby=\left(
\begin{array}{cc}
\bbI & 0 \\
0 & \bbU^{\top} \\
\end{array}
\right)\bby$ and $\widetilde \bbW=\left(
\begin{array}{cc}
0 & \mathcal{Y}^{\top}\Lambda \\
\Lambda\mathcal{Y} & 0 \\
\end{array}
\right)$, then we have
\begin{equation}\label{th10}
\bbx^{\top}\bbW^l\bby=\widetilde \bbx^{\top}\widetilde \bbW^l\widetilde \bby\,,
\end{equation}
where above diagonal entries of $\widetilde \bbW=(\widetilde w_{ij})_{1\le i,j\le n}$ are independent normal random variables such that for any positive integer $r$,
\begin{equation}\label{0903.1h}
\max_{1\le i,j\le n}\E|\widetilde{w}_{ij}|^r\le \|\Sigma\|^rc_r\,,
 \end{equation}where $c_r$ is the  $r$-th moment of standard normal distribution.   Actually, if $\{\widetilde w_{ij}\}_{1\le i,j\le n}$ were bounded random variables with
 \begin{equation}\label{th5}
 \max_{1\le i,j\le n}|\widetilde w_{ij}|\le 1\,,
 \end{equation}
then Lemmas 4 and 5 of \cite{FF18} imply that there exists a positive constant $c_l$ depending on $l$ such that
\begin{equation}\label{th8}
\E|\widetilde\bbx^{\top}(\widetilde \bbW^l-\E\widetilde \bbW^l)\widetilde\bby|^2\le c_l\sigma_n^{l-1}\,,
\end{equation}
and
\begin{equation}\label{th9}
|\E\widetilde\bbx^{\top}\widetilde \bbW^l\widetilde\bby|\le c_l\sigma_n^{l}\,.
\end{equation}

To establish Lemma \ref{mos}, it remains  to relax the bounded restriction \eqref{th5}. In other words, we need to replace the condition \eqref{th5} by the condition of $\widetilde w_{ij}$, $1\le i,j\le n$ in \eqref{0903.1h}.
 We highlight the difference of the proof.  Expanding $\E(\widetilde\bbx^\top\widetilde\bbW^l\widetilde\bby-\mathbb{E}\widetilde\bbx^\top\widetilde\bbW^l\widetilde\bby)^2$ yields
\begin{align} \label{th11}
& \E|\bbx^{\top}(\bbW^l-\E\bbW^l)\bby|^2=\E (\widetilde\bbx^\top\widetilde\bbW^l\widetilde\bby-\E\widetilde\bbx^\top\widetilde\bbW^l\widetilde\bby)^2\\
& =\sum_{1\le i_1,\cdots,i_{l+1},j_1,\cdots,j_{l+1}\le n,\atop
i_s\neq i_{s+1},\, j_s\neq j_{s+1},\, 1\le s\le l}\E \Big(\left(\widetilde x_{i_1}\widetilde w_{i_1i_2}\widetilde w_{i_2i_3}\cdots \widetilde w_{i_li_{l+1}}\widetilde y_{i_{l+1}}-\E\widetilde x_{i_1}\widetilde w_{i_1i_2}\widetilde w_{i_2i_3}\cdots \widetilde w_{i_li_{l+1}}\widetilde y_{i_{l+1}}\right) \nonumber \\
& \quad \times \left(\widetilde x_{j_1}\widetilde w_{j_1j_2}\widetilde w_{j_2j_3}\cdots\widetilde  w_{j_lj_{l+1}}\widetilde y_{j_{l+1}}-\E\widetilde x_{j_1}\widetilde w_{j_1j_2}\widetilde w_{j_2j_3}\cdots \widetilde w_{j_lj_{l+1}}\widetilde y_{j_{l+1}}\right)\Big)\,.\nonumber
\end{align}
Let $\bbi=(i_1,\ldots,i_{l+1})$ and $\bbj=(j_1,\ldots,j_{l+1})$ with $1\le i_1,\cdots,i_{l+1},j_1,\cdots,j_{l+1}\le n$, $i_s\neq i_{s+1},\, j_s\neq j_{s+1},\, 1\le s\le l$.   We define an undirected graph $\mathcal{G}_{\bbi}$ whose vertices represent $i_1,\ldots,i_{l+1}$  in $\bbi$, and only $i_{s}$ and $i_{s+1}$, for $s=1,\ldots, l$,  are connected in $\mathcal{G}_{\bbi}$. Similarly we can define $\mathcal{G}_{\bbj}$.  By the definitions of $\mathcal{G}_{\bbi}$ and $\mathcal{G}_{\bbj}$, for each term
  \begin{align*}
&\E \Big(\left(\widetilde x_{i_1}\widetilde w_{i_1i_2}\widetilde w_{i_2i_3}\cdots \widetilde w_{i_li_{l+1}}\widetilde y_{i_{l+1}}-\E\widetilde x_{i_1}\widetilde w_{i_1i_2}\widetilde w_{i_2i_3}\cdots \widetilde w_{i_li_{l+1}}\widetilde y_{i_{l+1}}\right) \nonumber \\
& \quad \times \left(\widetilde x_{j_1}\widetilde w_{j_1j_2}\widetilde w_{j_2j_3}\cdots\widetilde  w_{j_lj_{l+1}}\widetilde y_{j_{l+1}}-\E\widetilde x_{j_1}\widetilde w_{j_1j_2}\widetilde w_{j_2j_3}\cdots \widetilde w_{j_lj_{l+1}}\widetilde y_{j_{l+1}}\right)\Big)\,,
\end{align*}
there exists a one to one corresponding graph $\mathcal{G}_{\bbi}\cup\mathcal{G}_{\bbj}$ for $\{\widetilde w_{i_si_{s+1}}\}_{s=1}^l\cup\{\widetilde w_{j_sj_{s+1}}\}_{s=1}^l$.
If $\mathcal{G}_{\bbi}$ and $\mathcal{G}_{\bbj}$ are not connected, $\widetilde w_{i_1i_2}\widetilde w_{i_2i_3}\cdots \widetilde w_{i_li_{l+1}}$ and   $\widetilde w_{j_1j_2}\widetilde w_{j_2j_3}\cdots\widetilde  w_{j_lj_{l+1}}$ are independent, therefore we have
  \begin{align}
&\E \Big(\left(\widetilde x_{i_1}\widetilde w_{i_1i_2}\widetilde w_{i_2i_3}\cdots \widetilde w_{i_li_{l+1}}\widetilde y_{i_{l+1}}-\E\widetilde x_{i_1}\widetilde w_{i_1i_2}\widetilde w_{i_2i_3}\cdots \widetilde w_{i_li_{l+1}}\widetilde y_{i_{l+1}}\right) \\
& \quad \times \left(\widetilde x_{j_1}\widetilde w_{j_1j_2}\widetilde w_{j_2j_3}\cdots\widetilde  w_{j_lj_{l+1}}\widetilde y_{j_{l+1}}-\E\widetilde x_{j_1}\widetilde w_{j_1j_2}\widetilde w_{j_2j_3}\cdots \widetilde w_{j_lj_{l+1}}\widetilde y_{j_{l+1}}\right)\Big)=0\,.\nonumber
\end{align}
Therefore we have
\begin{align}\label{th6h}
\text{L.H.S. of }\eqref{th6}& =\sum_{\bbi,\bbj,  \mathcal{G}_{\bbi} \ \text{and} \ \mathcal{G}_{\bbj} \ \text{are connected},\atop
i_s\neq i_{s+1},\, j_s\neq j_{s+1},\, 1\le s\le l,}\E \Big(\left(\widetilde x_{i_1}\widetilde w_{i_1i_2}\widetilde w_{i_2i_3}\cdots \widetilde w_{i_li_{l+1}}\widetilde y_{i_{l+1}}-\E\widetilde x_{i_1}\widetilde w_{i_1i_2}\widetilde w_{i_2i_3}\cdots \widetilde w_{i_li_{l+1}}\widetilde y_{i_{l+1}}\right) \nonumber \\
& \quad \times \left(\widetilde x_{j_1}\widetilde w_{j_1j_2}\widetilde w_{j_2j_3}\cdots\widetilde  w_{j_lj_{l+1}}\widetilde y_{j_{l+1}}-\E\widetilde x_{j_1}\widetilde w_{j_1j_2}\widetilde w_{j_2j_3}\cdots \widetilde w_{j_lj_{l+1}}\widetilde y_{j_{l+1}}\right)\Big)\non
& \le \sum_{\bbi,\bbj,  \mathcal{G}_{\bbi} \ \text{and} \ \mathcal{G}_{\bbj} \ \text{are connected},\atop
i_s\neq i_{s+1},\, j_s\neq j_{s+1},\, 1\le s\le l,}\E|\widetilde x_{i_1}\widetilde w_{i_1i_2}\widetilde w_{i_2i_3}\cdots \widetilde w_{i_li_{l+1}}\widetilde y_{i_{l+1}}\widetilde x_{j_1}\widetilde w_{j_1j_2}\widetilde w_{j_2j_3}\cdots\widetilde  w_{j_lj_{l+1}}\widetilde y_{j_{l+1}}|\non
&+\sum_{\bbi,\bbj,  \mathcal{G}_{\bbi} \ \text{and} \ \mathcal{G}_{\bbj} \ \text{are connected},\atop
i_s\neq i_{s+1},\, j_s\neq j_{s+1},\, 1\le s\le l,}\E|\widetilde x_{i_1}\widetilde w_{i_1i_2}\widetilde w_{i_2i_3}\cdots \widetilde w_{i_li_{l+1}}\widetilde y_{i_{l+1}}|\E|\widetilde x_{j_1}\widetilde w_{j_1j_2}\widetilde w_{j_2j_3}\cdots\widetilde  w_{j_lj_{l+1}}\widetilde y_{j_{l+1}}|\,.
\end{align}

Notice that each expectation in the last two lines of \eqref{th6h} involves the product of independent random variables and the dependency of $\widetilde w_{i_1i_2}\widetilde w_{i_2i_3}\cdots \widetilde w_{i_li_{l+1}}$ and $\widetilde w_{j_1j_2}\widetilde w_{j_2j_3}\cdots\widetilde  w_{j_lj_{l+1}}$ are from some shared factors, say  $\widetilde w_{ab}^{m_1}$ and $\widetilde w_{ab}^{m_2}$ respectively, $m_1,m_2\ge 1$. By Holder's inequality that
$$\E|\widetilde w_{ab}|^{m_1}\E|\widetilde w_{ab}|^{m_2}\le \E|\widetilde w_{ab}|^{m_1+m_2}\,,$$
we have
\begin{equation}\label{0830.2h}
\eqref{th6h}\le 2\sum_{\bbi,\bbj,  \mathcal{G}_{\bbi} \ \text{and} \ \mathcal{G}_{\bbj} \ \text{are connected},\atop
i_s\neq i_{s+1},\, j_s\neq j_{s+1},\, 1\le s\le l,}\E|\widetilde x_{i_1}\widetilde w_{i_1i_2}\widetilde w_{i_2i_3}\cdots \widetilde w_{i_li_{l+1}}\widetilde y_{i_{l+1}}\widetilde x_{j_1}\widetilde w_{j_1j_2}\widetilde w_{j_2j_3}\cdots\widetilde  w_{j_lj_{l+1}}\widetilde y_{j_{l+1}}|\,.
\end{equation}
By \eqref{0830.2h}, to prove \eqref{th6}, it suffices to calculate the upper bound of the expectations at the right hand side of \eqref{0830.2h}. By the independency of $\widetilde w_{ij}$, the upper bound of $$\E|\widetilde x_{i_1}\widetilde w_{i_1i_2}\widetilde w_{i_2i_3}\cdots \widetilde w_{i_li_{l+1}}\widetilde y_{i_{l+1}}\widetilde x_{j_1}\widetilde w_{j_1j_2}\widetilde w_{j_2j_3}\cdots\widetilde  w_{j_lj_{l+1}}\widetilde y_{j_{l+1}}|$$ is controlled by the $r$-th moments of $\widetilde w_{ij}$ with \eqref{0903.1h}, $r=1,\ldots,2l$.  The topology of $\mathcal{G}_{\bbi} \ \text{and} \ \mathcal{G}_{\bbj}$ are the same as Lemma 4 of \cite{FF18}, the summation at the right hand side of \eqref{0830.2h} can be controlled by exactly the same steps as in  the proof of Lemma 4 in \cite{FF18}.  Hence \eqref{th6} can be proved following  the proof of Lemma 4 in \cite{FF18}. The proof of \eqref{th7} is similar to that of Lemma 5 in \cite{FF18} by the same modification.
 \end{proof}
The next Lemma follows directly from  Theorem 2.1 in \cite{alex2014}.
\begin{lem}\label{bod}
Under Assumption \ref{a2}, for any constant $c>1$, we have for any $\ep$, $D>0$, there exists an integer $n_0(\ep,D)$ depending on $\ep$ and $D$,  such that for all $n\ge n_0(\ep,D)$, it holds
$$\p\left(\|\bbW\|\ge c\max\{\|\Sigma\|,1\}( n^{\frac{1}{2}}+ p^{\frac{1}{2}})\right)\le n^{-D}\,.$$

\end{lem}
\begin{lem}\label{0505-1}
Suppose that $c_{12}=0$. If $n_1c_{11}\ge n_2c_{22}$, then we have
$$d_1^2=n_1c_{11}, \ d_2^2=n_2c_{22}\,,$$
otherwise
$$d_1^2=n_2c_{22}, \ d_2^2=n_1c_{11}\,,$$
\end{lem}
\begin{proof}
We prove this Lemma under the condition {\color{black} $n_1c_{11}\ge n_2c_{22}$ }.     Recall the definition of $\bbH$ in (\ref{0122.1}), if $c_{12}=0$, we have
$$\bbH=\bba_1\bba_1^{\top}c_{11}+\bba_2\bba_2^{\top}c_{22}.$$
Notice that $\bba_1^\top\bba_2=0$, $\|\bba_1\|_2^2=n_1$ and $\|\bba_2\|_2^2=n_2$, we imply that
$\frac{\bba_1}{\|\bba_1\|_2}$ and $\frac{\bba_2}{\|\bba_2\|_2}$ are the two eigenvectors of $\bbH$ with corresponding eigenvalues   $n_1c_{11}$ and $n_2c_{22}$. By the definition of $d_1$ and $d_2$ in (\ref{0122.1h}) and the condition that $n_1c_{11}\ge n_2c_{22}$, we have
$$d_1^2=n_1c_{11}, \ d_2^2=n_2c_{22}\,.$$
\end{proof}
\begin{lem}\label{0612-1}
Let $\bbA$ be a $p\times n$ matrix. Denote $\mathcal{A}=\left(
\begin{array}{cc}
0 & \bbA^{\top} \\
\bbA & 0 \\
\end{array}
\right)$. If $\lambda^2$ is a non-zero eigenvalue of $\bbA^\top\bbA$, then $\pm\lambda$ ($\lambda>0$) are the eigenvalues of $\mathcal{A}$. Moreover, assume that $\bba$ and $\bbb$ are the unit eigenvectors of $\bbA^\top\bbA$ and $\bbA\bbA^\top$ respectively corresponding to $\lambda^2$, then
\begin{equation}\label{0612.2}
\mathcal{A}\left(
\begin{array}{cc}
\bba \\
\bbb \\
\end{array}
\right)=\lambda\left(
\begin{array}{cc}
\bba \\
\bbb \\
\end{array}
\right), \ \mathcal{A}\left(
\begin{array}{cc}
\bba \\
-\bbb \\
\end{array}
\right)=-\lambda\left(
\begin{array}{cc}
\bba \\
-\bbb \\
\end{array}
\right)\,.
\end{equation}
\end{lem}
\begin{proof}
By the definition of eigenvalue, any eigenvalue of $\mathcal{A}$ (denoted by $x$) satisfy the following formula
\begin{equation}\label{0612.1}
\det(\mathcal{A}-x\bbI)=\det\left(\left(
\begin{array}{cc}
-x\bbI & \bbA^{\top} \\
\bbA & -x\bbI \\
\end{array}
\right)\right)=0\,.
\end{equation}
If $x\neq 0$, then \eqref{0612.1} is equivalent to
$$\det(\bbA^\top\bbA-x^2\bbI)=0\,.$$
Therefore the first conclusion that $\pm\lambda$ are the eigenvalues of $\mathcal{A}$.  By the definition of $\bba$ and $\bbb$, they are the right singular vector and left singular vector of $\bbA$ respectively corresponding to singular value $\lambda$. Then equations \eqref{0612.2} follow.
\end{proof}

\subsection{Proof of Lemma \ref{0508-1}}
 The high level idea for proving \eqref{0518.7} is to show that i) $\det(f(a_n))>0$ and $\det(f(b_n))>0$, ii) the function $\det(f(z))$ is strictly convex in $[a_n, b_n]$, and iii) there exists some $z\in (a_n,b_n)$ such that $\det(f(z))\leq 0$.  The result in \eqref{0518.8} is then proved by carefully analyzing the behavior of the function $\det(f(z))$ around $d_1$ and $d_2$.

We prove \eqref{0518.7} first. By the definition of  $f(z)$ in \eqref{eq2}, we have
\begin{eqnarray}\label{0211.2}
&&\det(f(z))=f_{11}(z)f_{22}(z)-f_{12}(z)f_{21}(z)\\
&&=\left(1+d_1\left(\mathcal{R}(\bbv_1,\bbv_1,z)-\mathcal{R}(\bbv_1,\bbV_{-},z)\big(-\bbD+\mathcal{R}(\bbV_{-},\bbV_{-},z)\big)^{-1}\mathcal{R}(\bbV_{-},\bbv_1,z)\right)\right)\non
&&\times\left(1+d_2\left(\mathcal{R}(\bbv_2,\bbv_2,z)-\mathcal{R}(\bbv_2,\bbV_{-},z)\big(-\bbD+\mathcal{R}(\bbV_{-},\bbV_{-},z)\big)^{-1}\mathcal{R}(\bbV_{-},\bbv_2,z)\right)\right)\non
&&-d_1d_2\left(\mathcal{R}(\bbv_1,\bbv_2,z)-\mathcal{R}(\bbv_1,\bbV_{-},z)\big(-\bbD+\mathcal{R}(\bbV_{-},\bbV_{-},z)\big)^{-1}\mathcal{R}(\bbV_{-},\bbv_2,z)\right)^2\,.\nonumber
\end{eqnarray}
By Lemma \ref{mos} and the expansion \eqref{0214.2}, for any $\bbM_1$ and $\bbM_2$  with finite columns and spectral norms, we have
	\begin{align}\label{1011.4hht}
  \left\|\mathcal{R}(\bbM_1,\bbM_2,z)+z^{-1}\bbM_1^\top\bbM_2\right\|=\|-\sum_{{l=2}}^L z^{-(l+1)}\bbM_1^\top\E\bbW^l\bbM_2\|=O(\sigma_n^2/a_n^3), \ z\in [a_n,b_n]\,,
	\end{align}
and
	\begin{align}\label{1011.4ht}
 \left\|\mathcal{R}'(\bbM_1,\bbM_2,z)-z^{-2}\bbM_1^\top\bbM_2\right\|=\|\sum_{{l=2}}^L (l+1)z^{-(l+2)}\bbM_1^\top\E\bbW^l\bbM_2\|=O(\sigma_n^2/a_n^4)\,.
	\end{align}

 Substituting $z = a_n$ into $f$, by \eqref{1011.4hht}, for large enough $n$ we have

   \begin{equation}\label{0825.2h}
   |\mathcal{R}(\bbv_1,\bbv_2,a_n)|=O\left(\frac{\sigma_n^2}{a_n^3}\right)\,
   \end{equation}
   \begin{equation}\label{0827.1h}
   \|\big(-\bbD+\mathcal{R}(\bbV_{-},\bbV_{-},z)\big)^{-1}\|=O(b_n)\ z\in [a_n,b_n]\,.
   \end{equation}
By \eqref{0825.2h} and \eqref{0827.1h} we have
  \begin{equation}\label{0826.4h}
  |\mathcal{R}(\bbv_i,\bbV_{-},z)\big(-\bbD+\mathcal{R}(\bbV_{-},\bbV_{-},z)\big)^{-1}\mathcal{R}(\bbV_{-},\bbv_j,z)|=O\left(\frac{\sigma_n^4}{a_n^5}\right), \ 1\le i,j\le 2, \ z\in [a_n,b_n]\,.
  \end{equation}
 By Assumption \ref{a2} on $\Sigma$, there exists a constant $c$ such that $\Sigma\ge c\bbI$, therefore we have
 \begin{equation}\label{0825.1h}
 \sigma_n^2\ge \max\{\bbv_1^\top\E\bbW^2\bbv_1,\bbv_2^\top\E\bbW^2\bbv_2\}\ge \min\{\bbv_1^\top\E\bbW^2\bbv_1,\bbv_2^\top\E\bbW^2\bbv_2\}\ge c\sigma_n^2.
 \end{equation}
 By \eqref{0825.1h} and Lemma \ref{mos}, for large enough $n$  we have
\begin{eqnarray*}
&&1+d_1\mathcal{R}(\bbv_1,\bbv_1,a_n)=1-\frac{d_1}{a_n}-\sum_{i\ge 2}^L\frac{d_1\bbv_1^\top\E\bbW^i\bbv_1}{a_n^{i+1}}\non
&&=1-\frac{d_1}{a_n}-\frac{d_1\bbv_1^\top\E\bbW^2\bbv_1}{a_n^{3}}+O(\frac{\sigma_n^3}{a_n^4})\le \frac{a_n-d_1}{2a_n}-\frac{c\sigma_n^2}{2a_n^2}\,,
\end{eqnarray*}
   and
   \begin{equation}\label{0825.3h}
   1+d_2\mathcal{R}(\bbv_2,\bbv_2,a_n)\le \frac{a_n-d_2}{2a_n}-\frac{c\sigma_n^2}{2a_n^2}\,.
   \end{equation}
  Substituting \eqref{0825.2h}--\eqref{0825.3h} into (\ref{0211.2}), we have
  \begin{equation}\label{0516.2}
  \det(f(a_n))>0\,.
  \end{equation}
   Similar to the proof from \eqref{0211.2} to \eqref{0516.2}, we  imply that
     \begin{equation}\label{0516.3}
  \det(f(b_n))>0\,.
  \end{equation}
   Moreover, by (\ref{0211.2}) and Lemma \ref{mos}, we imply that
  \begin{equation}\label{0211.3}
\Big(\det(f(z))\Big)''=-\frac{2d_1}{z^3}-\frac{2d_2}{z^3}+\frac{6d_1d_2}{z^4}+o\left(\frac{d_1d_2}{a_n^4}\right)>0, \ z\in [a_n,b_n]\,.
\end{equation}
Therefore $\det(f(z))$ is a strictly convex function  and has at most two solutions to the equation $\det(f(z))=0$,  $z\in[a_n,b_n]$.
 By \eqref{1011.4hht} and \eqref{1011.4ht}, we have
 \begin{align}\label{0825.4h}
 \frac{f_{11}'(z)}{d_1}&=\mathcal{R}'(\bbv_1,\bbv_1,z)-2\mathcal{R}'(\bbv_1,\bbV_{-},z)\big(-\bbD+\mathcal{R}(\bbV_{-},\bbV_{-},z)\big)^{-1}\mathcal{R}(\bbV_{-},\bbv_1,z)\\
 &-\mathcal{R}(\bbv_1,\bbV_{-},z)\left(\big(-\bbD+\mathcal{R}(\bbV_{-},\bbV_{-},z)\big)^{-1}\right)'\mathcal{R}(\bbV_{-},\bbv_1,z)>0, z\in[a_n,b_n]\,.\nonumber
 \end{align}
 Therefore $f_{11}(z)$ is a monotonic function in $[a_n,b_n]$. Moreover, by the definitions of $a_n$, $b_n$, $\sigma_n$ and Lemma \ref{mos}, we have
 $$f_{11}(a_n)<0, \ f_{11}(b_n)>0.$$
 Hence we conclude that there is a unique point $\tilde t_1\in[a_n,b_n]$ such that
 $$f_{11}(\tilde t_1)=0.$$
 By similar arguments and
  \begin{align}\label{0825.4t}
 \frac{f_{22}'(z)}{d_2}&=\mathcal{R}'(\bbv_2,\bbv_2,z)-2\mathcal{R}'(\bbv_2,\bbV_{-},z)\big(-\bbD+\mathcal{R}(\bbV_{-},\bbV_{-},z)\big)^{-1}\mathcal{R}(\bbV_{-},\bbv_2,z)\\
 &-\mathcal{R}(\bbv_2,\bbV_{-},z)\left(\big(-\bbD+\mathcal{R}(\bbV_{-},\bbV_{-},z)\big)^{-1}\right)'\mathcal{R}(\bbV_{-},\bbv_2,z)>0, z\in[a_n,b_n]\,,\nonumber
 \end{align}
  there exists $\tilde t_2\in[a_n,b_n]$ such that
 $$f_{22}(\tilde t_2)=0.$$
    Without loss of generality, we assume that \begin{equation}\label{0516.4}
 \tilde t_1\ge \tilde t_2\,.
 \end{equation}
It follows from (\ref{0211.2}) that
 \begin{equation}\label{0516.1}
 \det(f(\tilde t_1))\le 0 \  \text{and} \  \det(f(\tilde t_2))\le 0\,.
  \end{equation}
Therefore the existence of  $t_1$ and $t_2$ are ensured by \eqref{0516.2}, (\ref{0516.3}), \eqref{0516.1} and the convexity of $\det(f(z))$, $z\in[a_n,b_n]$ ($t_1$ is allowed to be equal to $t_2$). Furthermore, by the definition of $t_1$, $t_2$ and \eqref{0516.4} we have
\begin{equation}\label{0504.1}
b_n\ge t_1\ge \tilde t_1\ge \tilde t_2\ge t_2\ge a_n\,.
\end{equation}
Hence we complete the proof of \eqref{0518.7} and now we turn to \eqref{0518.8}. Calculating the first derivative of $f_{ii}$, by Lemma \ref{mos}, \eqref{0825.4h}  and \eqref{0825.4t} we have
\begin{equation}\label{0211.3}
f'_{ii}(z)=\frac{d_i}{z^2}+O\left(\frac{\sigma_n^2}{d^2_i}\right)\sim \frac{1}{d_i}\,, \ z \in[a_n,b_n]\,, i=1,2\,.
\end{equation}
Let $s_i=d_i+\frac{\E\bbv_1^\top\bbW^2\bbv_1}{d_i}$, for $f_{11}$, by Lemma \ref{mos} we have
$$f_{11}(s_1)=1-d_1\left(\frac{1}{s_1}+\frac{\bbv_1^\top\E\bbW^2\bbv_1}{s^3_1}\right)+O\left(\frac{\sigma_n^3}{d^3_1}\right)=O\left(\frac{\sigma_n^3}{d^3_1}\right)\,.$$
Combining this with (\ref{0211.3}), we imply that
\begin{equation}\label{0211.4}
\tilde t_1=d_1+\frac{\bbv_1^\top\E\bbW^2\bbv_1}{d_1}+O\left(\frac{\sigma_n^3}{d^2_1}\right)\,.
\end{equation}
Similarly, we also have
\begin{equation}\label{0211.5}
\tilde t_2=d_2+\frac{\bbv_2^\top\E\bbW^2\bbv_2}{d_2}+O\left(\frac{\sigma_n^3}{d^2_2}\right)\,.
\end{equation}
Finally, by Lemma \ref{mos} and \eqref{0211.2}, similar to the arguments of \eqref{0516.2} and \eqref{0516.3}, we have
\begin{equation}\label{0212.1}\det\left(f\left(d_1+\frac{2\bbv_1^\top\E\bbW^2\bbv_1}{d_1}+\frac{2\bbv_2^\top\E\bbW^2\bbv_2}{d_2}\right)\right)>0\,,\end{equation}
and
\begin{equation}\label{0212.2}\det\left(f\left(d_2-\frac{2\bbv_1^\top\E\bbW^2\bbv_1}{d_1}-\frac{2\bbv_2^\top\E\bbW^2\bbv_2}{d_2}\right)\right)>0\,.\end{equation}
By \eqref{0212.1} and \eqref{0212.2} and the convexity of $\det(f(z))$, we have
$$d_2-\frac{2\bbv_1^\top\E\bbW^2\bbv_1}{d_1}-\frac{2\bbv_2^\top\E\bbW^2\bbv_2}{d_2}\le t_2\le t_1\le d_1+\frac{2\bbv_1^\top\E\bbW^2\bbv_1}{d_1}+\frac{2\bbv_2^\top\E\bbW^2\bbv_2}{d_2}$$
Combining this with \eqref{0504.1}, \eqref{0211.4} and \eqref{0211.5}, we imply that
\begin{equation}\label{0518.6}
t_k-d_k=O\left(\frac{\sigma_n^2}{d_k}\right)\,, \ k=1,2\,,
\end{equation}
which implies Lemma \ref{0508-1} by \eqref{0513.4}.
\section{Discussion}\label{discu}

In this section, we discuss two directions to generalize our model. One is to enlarge the number of mixture components and the other is to allow non-gaussian distribution random vectors. Moreover, we also discuss the clustering boundary of our model under some additional restrictions in the last two sections.
\subsection{Three components in the mixture}
Suppose $Z$ follows a Gaussian mixture model that has three different populations means.
 $$Z\sim \pi_1 N(\mu_1,\Sigma)+\pi_2 N(\mu_2,\Sigma)+\pi_3 N(\mu_3,\Sigma)\,,$$
 where $\pi_1+\pi_2+\pi_3=1$.  Let a discrete random variable $Y$ be such that $\p(Y=k)=\pi_k$ and
 $$Z|Y=k\sim N(\mu_k,\Sigma)\,, \ k=1,2,3\,.$$
We define three $n$-dimensional vectors $\bba_k$, $k=1,2,3$, whose components are either $1$ or $0$. Concretely,
$$\bba_k(i)=1 \ \text{if and only if} \ X_i\sim N(\mu_k,\Sigma), \ k=1,2,3\,.$$
Moreover, we denote $n_k=\|\bba_k\|_2^2$ and $c_{kl}=\mu_k^\top\mu_l$, \ $1\le k,l\le 3$. Similar to the definition of $\bbH$ in (\ref{0122.1}), we define
\begin{equation}\label{0417.2}
\bbH\coloneqq(\E\bbX)^{\top}\E\bbX=\sum_{1\le k,l\le 3}\bba_k\bba_l^{\top}c_{kl}\ge 0\,.
\end{equation}
 By the same arguments as (\ref{0122.1})--(\ref{0417.1}), we conclude that $\bbH$ has a block structure. Let $\bbu$ be the unit eigenvector corresponding to one of the largest three eigenvalues of $\bbH$ and $d$ be the corresponding eigenvalue. Following similar arguments as in  (\ref{0417.1})--(\ref{0410.1}), we have that $\bbu$ has at most three distinct values.  Denote them by $v_k$, $k=1,2,3$, and we have
\begin{equation}\label{0417.3}
n_1c_{11}v_1+n_2c_{12}v_2+n_3c_{13}v_3=dv_1\,,
\end{equation}
\begin{equation}\label{0417.4}
n_1c_{12}v_1+n_2c_{22}v_2+n_3c_{23}v_3=dv_2\,,
\end{equation}
and
\begin{equation}\label{0417.5}
n_1c_{13}v_1+n_2c_{23}v_2+n_3c_{33}v_3=dv_3\,.
\end{equation}
The above equations imply that $d$ satisfy the following equation
\begin{eqnarray}
&&\left((d-n_2c_{22})(d-n_1c_{11})-n_1n_2c_{12}^2\right)\left((d-n_3c_{33})(d-n_1c_{11})-n_1n_3c_{13}^2\right)\\
&&=\left((d-n_1c_{11})n_3c_{23}+n_1n_3c_{12}c_{13}\right)\left((d-n_1c_{11})n_2c_{23}+n_1n_2c_{13}c_{23}\right)\,.\nonumber
\end{eqnarray}
The expression for $d$ will be more complicated than the two-component case we considered in this paper. It suggests the technical challenges that one would face to extend our current work to multiple-component Gaussian mixture models.

\subsection{Non-Gaussian distribution}
Checking the proof of our main theorem carefully, we can see that the key tool is Lemma \ref{mos}. As long as Lemma \ref{mos} holds, then all of our theorems holds.  Hence for non-gaussian distribution $Z$, it suffices to show Lemma \ref{mos} holds for non-gaussian distribution. The proof is expected to be more complicated than Lemmas 4 and 5 in \cite{FF18} and is worthy for further investigation.
\subsection{Clustering lower bound}
In this section, we investigate the clustering lower bound for our model when $p\sim n$. In addition, we impose Prior distribution on $Y_i$ -- assume that $\{Y_i\}$ are i.i.d., $Y_i\sim \text{Bernoulli}(1/2)$, $i=1,\ldots,n$. In addition, assume $\bmu_1 = -\bmu_2$. 
Let $l_i=2Y_i-1\in \{-1,1\}$ and $\widehat l_i$ be the estimator of $l_i$ by some clustering algorithm.  Similar to \cite{jin2017}, we introduce the Hamming distance to measure the performance of clustering:
\begin{equation}\label{gts6}
\text{Hamm}_n=\frac{1}{n}\inf_{s\in \{-1,1\}}\Big\{\sum_{i=1}^n\p(\widehat l_i\neq sl_i)\Big\}\,.
\end{equation}
The following theorem provides the clustering lower bound, below which clustering is impossible, regardless of what clustering method to use.

\begin{thm}\label{lowerbound}
 If $\bmu_1^T\bSig^{-1}\bmu_1\rightarrow 0$, then for any clustering approach, we have
\begin{align}\label{gts7}
 \lim\inf_{n\rightarrow \infty} \emph{Hamm}_n\ge \frac{1}{2}.
\end{align}
\end{thm}

\begin{proof}
The main idea of this proof largely follows from Theorem 1.1 of \cite{jin2017}. Notice that under the conditions of this Theorem, the model \eqref{0830.1h} becomes
\begin{equation}\label{gts8}
\bbx_i=l_i\bmu_1+\bbw_i,\ i=1,\ldots,n\,.
\end{equation}
 For any $1\le i\le n$, we consider the testing problem that
$$H_{-1}: l_i=-1 \ \text{vs} \ H_1:l_i=1\,.$$
Let $f_{\pm}^{(i)}$ be the joint density of $\bbX$ under $H_{\pm}$ respectively. By the property of total variation, it can be derived that
$$1-\|f_1-f_{-1}\|_{TV}\le \p(\widehat l_i\neq l_i|l_i=1)+\p(\widehat l_i\neq l_i|l_i=-1)\,.$$
By the assumption that $Y_i\sim \text{Bernoulli}(1/2)$ and $\|f_1-f_{-1}\|_{TV}=1/2\|f_1-f_{-1}\|_1$, we have
$$1/2-\frac{1}{4}\|f_1^{(i)}-f^{(i)}_{-1}\|_1\le \p(\widehat l_i\neq l_i)\,.$$
Therefore, in order to prove this theorem, it suffices to show that uniformly for all $1\le i \le n$, we have
$$\|f^{(i)}_1-f^{(i)}_{-1}\|_1\rightarrow 0\,.$$
Let $\mathbf{l}=(l_1,\ldots,l_n)^\top-l_i\bbe_i$. Then we have
\begin{align}
&\|f^{(i)}_1-f^{(i)}_{-1}\|_1=\E\Big|\int \text{sinh}(\bbx_i^\top\bSig^{-1}\bmu_1)e^{-\frac{\|\bSig^{-1/2}\bmu_1\|_2^2}{2}}e^{\mathbf{l}^\top\bbX^\top\bSig^{-1}\bmu_1-(n-1)\frac{\|\bSig^{-1/2}\bmu_1\|_2^2}{2}}dF(\mathbf{l})\Big|\non
&\le\int \E \Big|\text{sinh}(\bbx_i^\top\bSig^{-1}\bmu_1)e^{-\frac{\|\bSig^{-1/2}\bmu_1\|_2^2}{2}}e^{\mathbf{l}^\top\bbX^\top\bSig^{-1}\bmu_1-(n-1)\frac{\|\bSig^{-1/2}\bmu_1\|_2^2}{2}}\Big|dF(\mathbf{l})\,,
\end{align}
where $\E$ is the expectation under the distribution of $\bbX=\bbW$. Therefore, it suffices for us to show that
\begin{align}\label{gts9}
\E \Big|\text{sinh}(\bbx_i^\top\bSig^{-1}\bmu_1)e^{-\frac{\|\bSig^{-1/2}\bmu_1\|_2^2}{2}}e^{\mathrm{l}^\top\bbX^\top\bSig^{-1}\bmu_1-(n-1)\frac{\|\bSig^{-1/2}\bmu_1\|_2^2}{2}}\Big|\rightarrow 0\,.
\end{align}
Notice that $\bbx_i$ is independent of $\mathbf{l}^\top\bbX^\top$, we have
\begin{align}\label{gts10}
&\E \Big|\text{sinh}(\bbx_i^\top\bSig^{-1}\bmu_1)e^{-\frac{\|\bSig^{-1/2}\bmu_1\|_2^2}{2}}e^{\mathbf{l}^\top\bbX^\top\bSig^{-1}\bmu_1-(n-1)\|\bmu_1\|_2^2/2}\Big|\non
=&\E \Big|\text{sinh}(\bbx_i^\top\bSig^{-1}\bmu_1)e^{-\frac{\|\bSig^{-1/2}\bmu_1\|_2^2}{2}}\Big|\E[ e^{-\frac{\|\bSig^{-1/2}\bmu_1\|_2^2}{2}}e^{\mathrm{l}^\top\bbX^\top\bSig^{-1}\bmu_1-(n-1)\frac{\|\bSig^{-1/2}\bmu_1\|_2^2}{2}}]\non
=&e^{-\frac{\|\bSig^{-1/2}\bmu_1\|_2^2}{2}}\E \Big|\text{sinh}(\bbx_i^\top\bSig^{-1}\bmu_1)\Big|\,.
\end{align}
By the distribution of $l_i$ we know that \eqref{gts10} is independent of $i$.  Now we focus on $\E \Big|\text{sinh}(\bbx_i^\top\bSig^{-1}\bmu_1)\Big|$. Since the expectation is under the distribution that $\bbx_i=\bbw_i$,  $\bbx_i^\top\bSig^{-1}\bmu_1\sim N(0,\|\bSig^{-1/2}\bmu_1\|_2^2)$. For simplicity, let $z=\bbx_i^\top\bSig^{-1}\bmu_1/\|\bSig^{-1/2}\bmu_1\|_2$ and $\sigma=\|\bSig^{-1/2}\bmu_1\|_2$. Then
\begin{equation}\label{gts11}
  2\E \Big|\text{sinh}(\bbx_i^\top\bSig^{-1}\bmu_1)\Big|=2\E \Big|\text{sinh}(\sigma z)\Big|=2\int_{z\ge 0}\frac{e^{\sigma z}-e^{-\sigma z}}{\sqrt{2\pi}}e^{-z^2/2}dz
  \end{equation}
  \begin{equation}\label{gts12}
  \int_{z\ge 0}\frac{e^{\sigma z}}{\sqrt{2\pi}}e^{-z^2/2}dz=\frac{e^{\sigma^2/2}}{\sqrt{2\pi}}\int_{z\ge 0}e^{-(z-\sigma)^2/2}dz=e^{\sigma^2/2}\p(z\ge -\sigma)\,.
    \end{equation}
      \begin{equation}\label{gts13}
  \int_{z\ge 0}\frac{e^{-\sigma z}}{\sqrt{2\pi}}e^{-z^2/2}dz=\frac{e^{\sigma^2/2}}{\sqrt{2\pi}}\int_{z\ge 0}e^{-(z+\sigma)^2/2}dz=e^{\sigma^2/2}\p(z\ge \sigma)\,.
    \end{equation}
    By \eqref{gts12} and \eqref{gts13}, we imply that
  \begin{equation}\label{gts14}
  \E \Big|\text{sinh}(\bbx_i^\top\bSig^{-1}\bmu_1)\Big|=e^{\sigma^2/2}(\p(z\ge -\sigma)-\p(z\ge \sigma))=e^{\sigma^2/2}(\p(-\sigma\le z<\sigma)\,.
  \end{equation}
  By \eqref{gts10}, \eqref{gts14} and the condition that $ \|\bSig^{-1/2}\bmu_1\|_2\rightarrow 0$,
  we finish our proof.
  we finish our proof.
\end{proof}
\subsection{Exact recovery}
In this section, we consider a special case that $\bmu_1=-\bmu_2$. By Theorem \ref{thm: population-eg}, it is corresponding to the case that $d_2=0$ and $d_1^2=n_1c_{11}+n_2c_{22}=nc_{11}$. We prove that for a little bigger $\|\bmu_1\|$, we have the following theorem and Corollary \ref{cor12} for exact recovery.
\begin{thm}\label{exactrec}
Assume that $\Sigma=\bbI$, $\bmu_1=-\bmu_2$, $\|\bmu_1\|_{\infty}=O( \frac{1}{n^{1/4}})$,  $n=O(n_1)=O(n_2)$ and $p\sim n$, if there exists a positive constant $\ep$ such that $c_{11}\ge 2(1+\ep)\log n$, then there exists $s\in\{\pm 1\}$ such that with probability tending to 1, we have
\begin{equation}\label{gta1}
  \sqrt n\min_{1\le i\le n}\{sl_i\widehat\bbu_1(i)\}\ge  1-\frac{1}{\sqrt{1+\ep}}-\frac{C}{\sqrt{\log n}}\,,
\end{equation}
for some positive constant $\bbC$.
\end{thm}
\begin{proof}
We prove this theorem by considering the linearization matrix $\mathcal{Z}$ and  $\widehat \bbv_1$. The idea of the proof follows from the proof of Theorem 3.1 of \cite{abbe2017entrywise}. Concretely, we prove that \textbf{A1}--\textbf{A4} of \cite{abbe2017entrywise} hold and apply Theorem 1.1 of \cite{abbe2017entrywise} to show our result. Substituting $d_1^2=nc_{11}$ and $c_{11}=c_{22}=-c_{12}$ into \eqref{0122.5} and \eqref{0122.6}, without loss of generality, assume $\bbu_1$ has two different values $v_1$ and $v_2$ such that
$$v_1=-v_2=\frac{1}{\sqrt n}\,,$$
where $v_1$ is corresponding to $Y_i=1$ and $v_2$ is corresponding to $Y_i=0$. Then we have
\begin{equation}\label{gta2}
l_i\bbu_1(i)=\frac{1}{\sqrt n}, \ i=1,\ldots,n\,.
\end{equation}
By Lemma \ref{bod}, for any positive constant $c>1$, $D$  and sufficiently large $n$ we have
$$\p\left(\|\bbW\|\ge c( \sqrt n+\sqrt p)\right)\le n^{-D}\,.$$
Setting $\gamma=\max\{\frac{\|\bmu_1\|_{\infty}}{\sqrt{\log n}},\frac{1}{\sqrt n}\}\rightarrow 0$, we have 
\begin{equation}\label{gta3}
\max\{\sqrt{c_{11}},\|\bmu_1\|_{\infty}\sqrt n\}\le \gamma d_1\,.
\end{equation}
Notice that  $\mathcal{Z}$ and $\E\mathcal{Z}$ are corresponding to $\bbA$ and $\bbA^*$ of \cite{abbe2017entrywise}. Let $\Delta^*=d_1$, by \eqref{gta3}, \textbf{A1} of \cite{abbe2017entrywise} holds. Moreover, \textbf{A2} follows from the assumption that $\Sigma=\bbI$.
 By Lemma \ref{bod}, it is easy to see that \textbf{A3} of \cite{abbe2017entrywise} holds  by \eqref{gta2}.
Similar to the proof of Theorem 3.1 in \cite{abbe2017entrywise}, \textbf{A4} holds by setting $\phi(x)=x$.  By Theorem 1.1 of \cite{abbe2017entrywise}, with probability tending to 1, there exists a positive constant $C$ such that
\begin{equation}\label{gta4}
\min_{s\in \{\pm 1\}}\|s\widehat\bbv_1-\frac{\mathcal{Z}\bbv_1}{d_1}\|_{\infty}=\min_{s\in \{\pm 1\}}\|s\widehat\bbv_1-\bbv_1-\frac{(\mathcal{Z}-\E\mathcal{Z})\bbv_1}{d_1}\|_{\infty}\le C \gamma\|\bbv_1\|_{\infty}\,,
\end{equation}
where $\bbv_1$ is the eigenvector of $\E\mathcal{Z}$ corresponding to $d_1$.
By Lemma \ref{0612-1}, we have $\bbv_1=\frac{1}{\sqrt 2}(\bbu_1^\top,\frac{\bmu_1^\top}{c_{11}})^\top$. Therefore by the conditions that $\|\bmu_1\|_{\infty}=O( \frac{1}{n^{1/4}})$ and $n=O(n_1)=O(n_2)=O(p)$, we have
\begin{equation}\label{gta5}
\gamma\|\bbv_1\|_{\infty}=O(\frac{1}{\sqrt{n\log n}})\,.\end{equation}
Notice that each entry of $\sqrt 2(\mathcal{Z}-\E\mathcal{Z})\bbv_1$ follows  a standard gaussian distribution. This implies that
\begin{equation}\label{gta6}
\p(\max_{1\le i\le n}|\bbe_i^\top(\mathcal{Z}-\E\mathcal{Z})\bbv_1|\ge \sqrt{\log n})=O(\frac{1}{\sqrt{\log n}})\,.
\end{equation}
By  \eqref{gta4}--\eqref{gta6}, with probability tending to $1$, there exists $s\in \{\pm 1\}$ and some positive constant $C$ such that
\begin{equation}\label{gta8a}
  \sqrt n\max_{1\le i\le n}\{\|s\widehat\bbv_1(i)-\bbv_1(i)\|_{\infty}\}\le \frac{C}{\sqrt{2n\log n}}+\frac{\sqrt{\log n}}{\sqrt{2(1+\ep)n\log n}}\,.
\end{equation}
Notice that $\bbv_1=\frac{1}{\sqrt 2}(\bbu_1^\top,\frac{\bmu_1^\top}{c_{11}})^\top$ and the first $n$ entries of  $\widehat\bbv_1$ is $\frac{1}{\sqrt 2}\widehat\bbu_1$, by \eqref{gta2} and \eqref{gta8a}, we have

\begin{equation}\label{gta7}
  \sqrt n\min_{1\le i\le n}\{sl_i\widehat\bbu_1(i)\}\ge 1-\frac{1}{\sqrt{1+\ep}}-\frac{C}{\sqrt{\log n}}\,.
\end{equation}
\end{proof}

By Theorem \ref{exactrec}, we have the following corollary to ensure the existence of exact recovery for the model.

\begin{coro}\label{cor12}
Under the conditions of Theorem \ref{exactrec}, there exists one clustering approach such that
\begin{equation}\label{gta8}\p(\widehat Y_i=Y_i, i=1\ldots,n)=1-o(1)\,.
\end{equation}
\end{coro}
\begin{proof}
The following clustering procedure suffices.

1. Calculate the eigenvector of $\mathcal{Z}$ corresponding to the largest eigenvalue, which is $\widehat\bbv_1$ as we defined before.

2. $\widehat Z_i=\frac{sgn(\widehat\bbv_1(i))+1}{2}, i=1,\ldots,n$.

If $\sum_{i=1}^n (2\widehat Z_i-1)l_i>0$, we let $\widehat Y_i=\widehat Z_i$, otherwise  $\widehat Y_i=-(\widehat Z_i-1)$. Without loss of generality, we assume that $\sum_{i=1}^n(2 \widehat  Z_i-1)l_i>0$ and therefore $\widehat Y_i=\widehat Z_i$.
By the definition of $\widehat Z_i$ and the condition that $\sum_{i=1}^n(2 \widehat  Z_i-1)l_i>0$, Theorem \ref{exactrec} holds for $s=1$. Hence
\begin{equation}\label{gtb1}
\p(\widehat Y_i=Y_i, i=1\ldots,n|\sum_{i=1}^n(2 \widehat  Z_i-1)l_i>0)=1-o(1).
\end{equation}
By almost the same arguments, we can prove similarly that
\begin{equation}\label{gtb2}
\p(\widehat Y_i=Y_i, i=1\ldots,n|\sum_{i=1}^n(2 \widehat  Z_i-1)l_i\le 0)=1-o(1).
\end{equation}
Therefore, \eqref{gta8} follows from \eqref{gtb1} and \eqref{gtb2}.
\end{proof}
\newpage

\bibliographystyle{unsrtnat}

\bibliography{references}

\end{document}